\documentclass[final,hidelinks]{dmtcs-episciences}
\usepackage{amssymb, amsmath, amsthm, latexsym}
\usepackage{eucal}
\usepackage{ytableau}
\usepackage{undertilde}
\usepackage{cancel}
\usepackage{caption}
\usepackage[pdftex]{graphicx}

\newtheorem{thm}{Theorem}[section]
\newtheorem{cor}[thm]{Corollary}
\newtheorem{lem}[thm]{Lemma}
\newtheorem{prop}[thm]{Proposition}

\newtheorem{fact}[thm]{Fact}

\theoremstyle{definition}

\theoremstyle{remark}

\numberwithin{equation}{section}

\begin{document}

\title[Demazure tableaux sets' convexity, parabolic Catalan numbers]{Convexity of tableau sets for type A Demazure characters (key polynomials), parabolic Catalan numbers}

\author{Robert A. Proctor\affiliationmark{1} \and Matthew J. Willis\affiliationmark{2}}

\affiliation{University of North Carolina, Chapel Hill, NC 27599 U.S.A. \\
Wesleyan University, Middletown, CT 06457 U.S.A.}


\received{2017-12-17}

\revised{2018-6-13}

\accepted{2018-7-8}

\publicationdetails{20}{2018}{2}{3}{4158}

\maketitle

\begin{abstract}

This is the first of three papers that develop structures which are counted by a ``parabolic'' generalization of Catalan numbers.  Fix a subset  $R$  of  $\{1,2,..,n-1 \}$.  Consider the ordered partitions of  $\{1,2,..,n\}$  whose block sizes are determined by  $R$.  These are the ``inverses'' of (parabolic) multipermutations whose multiplicities are determined by  $R$.  The standard forms of the ordered partitions are referred to as ``$R$-permutations''.  The notion of  312-avoidance is extended from permutations to $R$-permutations.

Let  $\lambda$  be a partition of  $N$  such that the set of column lengths in its shape is  $R$  or  $R \cup \{n\}$.  Fix an  $R$-permutation   $\pi$.  The type A Demazure character (key polynomial) in  $x_1, ..., x_n$  that is indexed by  $\lambda$  and  $\pi$  can be described as the sum of the weight monomials for some of the semistandard Young tableau of shape  $\lambda$  that are used to describe the Schur function indexed by  $\lambda$.  Descriptions of these ``Demazure'' tableaux developed by the authors in earlier papers are used to prove that the set of these tableaux is convex in  $\mathbb{Z}^N$  if and only if  $\pi$  is  $R$-312-avoiding if and only if the tableau set is the entire principal ideal generated by the key of $\pi$.  These papers were inspired by results of Reiner and Shimozono and by Postnikov and Stanley concerning coincidences between Demazure characters and flagged Schur functions.  This convexity result is used in the next paper to deepen those results from the level of polynomials to the level of tableau sets.

The  $R$-parabolic Catalan number is defined to be the number of  $R$-312-avoiding permutations.  These special  $R$-permutations are reformulated as  ``$R$-rightmost clump deleting'' chains of subsets of $\{1, 2, ..., n \}$ and as ``gapless  $R$-tuples'';  the latter  $n$-tuples arise in multiple contexts in these papers.

\end{abstract}

\vspace{1pc}\noindent\textbf{MSC Codes.}  {05E10, 05A05, 14M15, 52B12}

\section{Introduction}

\noindent This is the first of three papers that develop and use structures which are counted by a ``parabolic'' generalization of the Catalan numbers.  Apart from some motivating remarks,  it can be read by anyone interested in tableaux.  It is self-contained,  except for a few references to its tableau precursors \cite{Wi2} and \cite{PW1}.  Fix  $n \geq 1$  and set  $[n-1] := \{1,2,...,n-1\}$.  Choose a subset  $R \subseteq [n-1]$  and set  $r := |R|.$  The section on  $R$-Catalan numbers can be understood as soon as a few definitions have been read.  Our ``rightmost clump deleting'' chains of sets defined early in Section 4 became Exercise 2.202 in Stanley's list \cite{Sta} of interpretations of the Catalan numbers.

Consider the ordered partitions of the set $[n]$  with  $r+1$  blocks of fixed sizes that are determined by using  $R$  to specify ``dividers''.  These ordered partitions can be viewed as being the ``inverses'' of multipermutations whose  $r+1$  multiplicities are determined by  $R$.  Setting  $J := [n-1] \backslash R$,  these multipermutations depict the minimum length coset representatives forming  $W^J$  for the quotient of  $S_n$  by the parabolic subgroup  $W_J$.  We refer to the standard forms of the ordered partitions as ``$R$-permutations''.  When  $R = [n-1]$,  the  $R$-permutations are just the permutations of  $[n]$.  The number of  312-avoiding permutations of  $[n]$  is the  $n^{th}$ Catalan number.  In 2012 we generalized the notion of 312-pattern avoidance for permutations to that of ``$R$-312-avoidance'' for  $R$-permutations.  Here we define the ``parabolic  $R$-Catalan number'' to be the number of  $R$-312-avoiding  $R$-permutations.

Let  $N \geq 1$  and fix a partition  $\lambda$  of  $N$.  The shape of  $\lambda$  has  $N$  boxes;  assume that it has at most  $n$  rows.  Let  $\mathcal{T}_\lambda$  denote the set of semistandard Young tableaux of shape  $\lambda$  with values from  $[n]$.  The content weight monomial  $x^{\Theta(T)}$  of a tableau  $T$ in $\mathcal{T}_\lambda$ is formed from the census  $\Theta(T)$  of the values from  $[n]$  that appear in  $T$.  The Schur function in  $x_1, ..., x_n$  indexed by  $\lambda$  can be expressed as the sum over  $T$ in $\mathcal{T}_\lambda$  of the content weight monomials  $x^{\Theta(T)}$.  Let  $R_\lambda \subseteq [n-1]$  be the set of column lengths in the shape  $\lambda$  that are less than  $n$.  The type A Demazure characters (key polynomials) in  $x_1, ..., x_n$  can be indexed by pairs  $(\lambda,\pi)$,  where  $\lambda$  is a partition as above and  $\pi$  is an  $R_\lambda$-permutation.  We refer to these as ``Demazure polynomials''.  The Demazure polynomial indexed by  $(\lambda,\pi)$  can be expressed as the sum of the monomials  $x^{\Theta(T)}$  over a set  $\mathcal{D}_\lambda(\pi)$  of ``Demazure tableaux'' of shape  $\lambda$  for the  $R$-permutation  $\pi$.

Regarding  $\mathcal{T}_\lambda$  as a poset via componentwise comparison,  it can be seen that the principal order ideals $[T]$ in  $\mathcal{T}_\lambda$   form convex polytopes in  $\mathbb{Z}^N$.  The set $\mathcal{D}_\lambda(\pi)$ can be seen to be a certain subset of the ideal $[Y_\lambda(\pi)]$, where the tableau $Y_\lambda(\pi)$ is the ``key'' of $\pi$.  It is natural to ask for which $R$-permutations $\pi$ one has $\mathcal{D}_\lambda(\pi) = [Y_\lambda(\pi)]$.  Our first main result is:  If  $\pi$  is an  $R_\lambda$-312-avoiding  $R_\lambda$-permutation,  then the tableau set  $\mathcal{D}_\lambda(\pi)$  is all of the principal ideal $[Y_\lambda(\pi)]$  (and hence is convex in  $\mathbb{Z}^N$).  Our second main result is conversely:  If  $\mathcal{D}_\lambda(\pi)$  forms a convex polytope in  $\mathbb{Z}^N$  (this includes the principal ideals $[Y_\lambda(\pi)]$),  then the  $R_\lambda$-permutation  $\pi$  is  $R_\lambda$-312-avoiding.  So we can say exactly when one has $\mathcal{D}_\lambda(\pi) = [Y_\lambda(\pi)]$.  Our earlier papers \cite{Wi2} and \cite{PW1} gave the first tractable descriptions of the Demazure tableau sets  $\mathcal{D}_\lambda(\pi)$.  Those results provide the means to prove the main results here.

Demazure characters arose in 1974 when Demazure introduced certain  $B$-modules while studying singularities of Schubert varieties in the  $G/B$  flag manifolds.  Flagged Schur functions arose in 1982 when Lascoux and Sch{\"u}tzenberger were studying Schubert polynomials for the flag manifold  $GL(n)/B$.  Like the Demazure polynomials,  the flagged Schur functions in  $x_1, ..., x_n$  can be expressed as sums of the weight monomial  $x^{\Theta(T)}$  over certain subsets of  $\mathcal{T}_\lambda$.  Reiner and Shimozono \cite{RS} and then Postnikov and Stanley \cite{PS} described coincidences between the Demazure polynomials and the flagged Schur functions.  Beginning in 2011,  our original motivation for this project was to better understand their results.  In the second paper \cite{PW3} in this series,  we deepen their results:  Rather than obtaining coincidences at the level of polynomials,  we employ the main results of this paper to obtain the coincidences at the underlying level of the tableau sets that are used to describe the polynomials.  Fact \ref{fact320.3}, Proposition \ref{prop320.2}, and Theorem \ref{theorem340} are also needed in \cite{PW3}.  In Section 8 we indicate why our characterization of convexity for the sets $\mathcal{D}_\lambda(\pi)$ may be of interest in algebraic geometry and representation theory.

Each of the two main themes of this series of papers is at least as interesting to us as is any one of the stated results in and of itself.  One of these themes is that the structures used in the three papers are counted by numbers that can be regarded as being ``parabolic'' generalizations of the Catalan numbers.  In these three papers these structures are respectively introduced to study convexity, the coincidences, and to solve a problem concerning the ``nonpermutable'' property of Gessel and Viennot.  It turned out that by 2014,  Godbole, Goyt, Herdan, and Pudwell had independently introduced \cite{GGHP} a general notion of pattern avoidance for ordered partitions that includes our notion of  $R$-312-avoidance for  $R$-permutations.  Apparently their motivations for developing their definition were purely enumerative.  Chen, Dai, and Zhou obtained \cite{CDZ} further enumerative results.  As a result of the work of these two groups,  two sequences were added to the OEIS.  As is described in our last section, one of those is formed from the counts considered here for a sequence of particular cases.  The other of those is formed by summing the counts considered here for all cases.  In our series of papers,  the parabolic Catalan count arises ``in nature'' in a number of interrelated ways.  In this first paper this quantity counts ``gapless  $R$-tuples'',  ``$R$-rightmost clump deleting chains'', and convex Demazure tableau sets.  The parabolic Catalan number further counts roughly another dozen structures in our two subsequent papers.  After the first version \cite{PW2} of this paper was initially distributed,  we learned of a different (but related) kind of parabolic generalization of the Catalan numbers due to M{\"u}hle and Williams.  This is described at the end of this paper.

The other main theme of this series of papers is the ubiquity of some of the structures that are counted by the parabolic Catalan numbers.  The gapless  $R$-tuples arise as the images of the  $R$-312-avoiding $R$-permutations under the  $R$-ranking map in this paper and as the minimum members of the equivalence classes for the indexing $n$-tuples of a generalization of the flagged Schur functions in our second paper.  Moreover,  the  $R$-gapless condition provides half of the solution to the nonpermutability problem considered in our third paper \cite{PW4}.  Since the gapless  $R$-tuples and the structures equivalent to them are enumerated by a parabolic generalization of Catalan numbers,  it would not be surprising if they were to arise in further contexts.

The material in this paper first appeared as one-third of the overly long manuscript \cite{PW2}.  The second paper \cite{PW3} in this series presents most of the remaining material from \cite{PW2}.  Section 11 of \cite{PW3} describes the projecting and lifting processes that relate the notions of 312-avoidance and of $R$-312-avoidance.

Definitions are presented in Sections 2 and 3.  In Section 4 we reformulate the  $R$-312-avoiding  $R$-permutations as  $R$-rightmost clump deleting chains and as gapless  $R$-tuples.  To prepare for the proofs of our two main results,  in Section 5 we associate certain tableaux to these structures.  Our main results are presented in Sections 6 and 7.  Section 8 indicates why convexity for the sets of $\mathcal{D}_\lambda(\pi)$ may be of further interest,  and Section 9 contains remarks on enumeration.

\section{General definitions and definitions of $\mathbf{\emph{R}}$-tuples}

In posets we use interval notation to denote principal ideals and convex sets.  For example, in $\mathbb{Z}$ one has $(i, k] = \{i+1, i+2, ... , k\}$.  Given an element $x$ of a poset $P$, we denote the principal ideal $\{ y \in P : y \leq x \}$ by $[x]$.  When $P = \{1 < 2 < 3 < ... \}$, we write $[1,k]$ as $[k]$.  If $Q$ is a set of integers with $q$ elements, for $d \in [q]$ let $rank^d(Q)$ be the $d^{th}$ largest element of $Q$.  We write $\max(Q) := rank^1(Q)$ and $\min(Q) := rank^q(Q)$.  A set $\mathcal{D} \subseteq \mathbb{Z}^N$ for some $N \geq 1$ is a \textit{convex polytope} if it is the solution set for a finite system of linear inequalities.

Fix $n \geq 1$ throughout the paper.  Except for $\zeta$, various lower case Greek letters indicate various kinds of $n$-tuples of non-negative integers. Their entries are denoted with the same letter.  An $nn$-\textit{tuple} $\nu$ consists of $n$ \emph{entries} $\nu_i \in [n]$ that are indexed by \emph{indices} $i \in [1,n]$.  An $nn$-tuple $\phi$ is a \textit{flag} if $\phi_1 \leq \ldots \leq \phi_n$.  An \emph{upper tuple} is an $nn$-tuple $\upsilon$ such that $\upsilon_i \geq i$ for $i \in [n]$.  The upper flags are the sequences of the $y$-coordinates for the above-diagonal Catalan lattice paths from $(0, 0)$ to $(n, n)$.  A \emph{permutation} is an $nn$-tuple that has distinct entries.  Let $S_n$ denote the set of permutations.  A permutation $\pi$ is $312$-\textit{avoiding} if there do not exist indices $1 \leq a < b < c \leq n$ such that $\pi_a > \pi_b < \pi_c$ and $\pi_a > \pi_c$.  (This is equivalent to its inverse being 231-avoiding.)  Let $S_n^{312}$ denote the set of 312-avoiding permutations.  By Exercises 116 and 24 of \cite{Sta}, these permutations and the upper flags are counted by the Catalan number $C_n := \frac{1}{n+1}\binom{2n}{n}$.

Fix $R \subseteq [n-1]$ through the end of Section 7.  Denote the elements of $R$ by $q_1 < \ldots < q_r$ for some $r \geq 0$.  Set $q_0 := 0$ and $q_{r+1} := n$.  We use the $q_h$ for $h \in [r+1]$ to specify the locations of $r+1$ ``dividers'' within $nn$-tuples:  Let $\nu$ be an $nn$-tuple.  On the graph of $\nu$ in the first quadrant draw vertical lines at $x = q_h + \epsilon$ for $h \in [r+1]$ and some small $\epsilon > 0$.  In Figure 7.1 we have $n = 9$ and $R = \{ 2, 3, 5, 7 \}$.  These $r+1$ lines indicate the right ends of the $r+1$ \emph{carrels} $(q_{h-1}, q_h]$ \emph{of $\nu$} for $h \in [r+1]$.  An \emph{$R$-tuple} is an $nn$-tuple that has been equipped with these $r+1$ dividers.  Fix an $R$-tuple $\nu$;  we portray it by $(\nu_1, ... , \nu_{q_1} ; \nu_{q_1+1}, ... , \nu_{q_2}; ... ; \nu_{q_r+1}, ... , \nu_n)$.  Let $U_R(n)$ denote the set of upper $R$-tuples.  Let $UF_R(n)$ denote the subset of $U_R(n)$ consisting of upper flags.  Fix $h \in [r+1]$.  The $h^{th}$ carrel has $p_h := q_h - q_{h-1}$ indices.  The $h^{th}$ \emph{cohort} of $\nu$ is the multiset of entries of $\nu$ on the $h^{th}$ carrel.

An \emph{$R$-increasing tuple} is an $R$-tuple $\alpha$ such that  $\alpha_{q_{h-1}+1} < ... < \alpha_{q_h}$ for $h \in [r+1]$.  Let $UI_R(n)$ denote the subset of $U_R(n)$ consisting of $R$-increasing upper tuples.  Consult Table 2.1 for an example and a nonexample.  Boldface entries indicate failures. It can be seen that $|UI_R(n)| = n! / \prod_{h=1}^{r+1} p_h!  =: \binom{n}{R}$.  An $R$-\textit{permutation} is a permutation that is $R$-increasing when viewed as an $R$-tuple.  Let $S_n^R$ denote the set of $R$-permutations.  Note that $| S_n^R| = \binom{n}{R}$.  We refer to the cases $R = \emptyset$ and $R = [n-1]$ as the \emph{trivial} and \emph{full cases} respectively.  Here $| S_n^\emptyset | = 1$ and $| S_n^{[n-1]} | = n!$ respectively.  An $R$-permutation $\pi$ is $R$-$312$-\textit{containing} if there exists $h \in [r-1]$ and indices $1 \leq a \leq q_h < b \leq q_{h+1} < c \leq n$ such that $\pi_a > \pi_b < \pi_c$ and $\pi_a > \pi_c$.  An $R$-permutation is $R$-$312$-\textit{avoiding} if it is not $R$-$312$-containing.  (This is equivalent to the corresponding multipermutation being 231-avoiding.)  Let $S_n^{R\text{-}312}$ denote the set of $R$-312-avoiding permutations.  We define the \emph{$R$-parabolic Catalan number} $C_n^R$ by $C_n^R := |S_n^{R\text{-}312}|$.

\begin{figure}[h!]
\begin{center}
\begin{tabular}{lccc}
\underline{Type of $R$-tuple} & \underline{Set} & \underline{Example} & \underline{Nonexample} \\ \\
$R$-increasing upper tuple & $\alpha \in UI_R(n)$ & $(2,6,7;4,5,7,8,9;9)$ & $(3,5,\textbf{5};6,\textbf{4},7,8,9;9)$ \\ \\
$R$-312-avoiding permutation & $\pi \in S_n^{R\text{-}312}$  & $(2,3,6;1,4,5,8,9;7)$ & $(2,4,\textbf{6};1,\textbf{3},7,8,9;\textbf{5})$ \\ \\
Gapless $R$-tuple & $\gamma \in UG_R(n)$ & $(2,4,6;4,5,6,7,9;9)$ & $(2,4,6;\textbf{4},\textbf{6},7,8,9;9)$ \\ \\
\end{tabular}
\caption*{Table 2.1.  (Non-)Examples of R-tuples for $n = 9$ and $R = \{3,8\}$.}
\end{center}
\end{figure}

Next we consider $R$-increasing tuples with the following property:  Whenever there is a descent across a divider between carrels, then no ``gaps'' can occur until the increasing entries in the new carrel ``catch up''.  So we define a \emph{gapless $R$-tuple} to be an $R$-increasing upper tuple $\gamma$ such that whenever there exists $h \in [r]$ with $\gamma_{q_h} > \gamma_{q_h+1}$, then $s := \gamma_{q_h} - \gamma_{q_h+1} + 1 \leq p_{h+1}$ and the first $s$ entries of the $(h+1)^{st}$ carrel $(q_h, q_{h+1} ]$ are $\gamma_{q_h}-s+1, \gamma_{q_h}-s+2, ... , \gamma_{q_h}$.  The failure in Table 2.1 occurs because the absence of the element $5 \in [9]$ from the second carrel creates a gap.  Let $UG_R(n) \subseteq UI_R(n)$ denote the set of gapless $R$-tuples.  Note that a gapless $\gamma$ has $\gamma_{q_1} \leq \gamma_{q_2} \leq ... \leq \gamma_{q_r} \leq \gamma_{q_{r+1}}$.  So in the full $R = [n-1]$ case, each gapless $R$-tuple is a flag.  Hence $UG_{[n-1]}(n) = UF_{[n-1]}(n)$.

An $R$-\textit{chain} $B$ is a sequence of sets $\emptyset =: B_0 \subset B_1 \subset \ldots \subset B_r \subset B_{r+1} := [n]$ such that $|B_h| = q_h$ for $h \in [r]$.  A bijection from $R$-permutations $\pi$ to $R$-chains $B$ is given by $B_h := \{\pi_1, \pi_2, \ldots, \pi_{q_h}\}$ for $h \in [r]$.  We indicate it by $\pi \mapsto B$.  The $R$-chains for the two $R$-permutations appearing in Table 2.1 are $\emptyset \subset \{2, 3, 6\} \subset \{ 1,2,3,4,5,6,8,9 \} \subset [9]$ and $\emptyset \subset \{2, 4, 6\} \subset \{1,2,3,4,6,7,8,9\} \subset [9]$.  Fix an $R$-permutation $\pi$ and let $B$ be the corresponding $R$-chain.  For $h \in [r+1]$, the set $B_h$ is the union of the first $h$ cohorts of $\pi$.  Note that $R$-chains $B$ (and hence $R$-permutations $\pi$) are equivalent to the $\binom{n}{R}$ objects that could be called ``ordered $R$-partitions of $[n]$''; these arise as the sequences $(B_1 \backslash B_0, B_2\backslash B_1, \ldots, B_{r+1}\backslash B_r)$ of $r+1$ disjoint nonempty subsets of sizes $p_1, p_2, \ldots, p_{r+1}$.  Now create an $R$-tuple $\Psi_R(\pi) =: \psi$ as follows:  For $h \in [r+1]$ specify the entries in its $h^{th}$ carrel by $\psi_i := \text{rank}^{q_h-i+1}(B_h)$ for $i \in (q_{h-1},q_h]$.  For a model, imagine there are $n$ discus throwers grouped into $r+1$ heats of $p_h$ throwers for $h \in [r+1]$.  Each thrower gets one throw, the throw distances are elements of $[n]$, and there are no ties.  After the $h^{th}$ heat has been completed, the $p_h$ longest throws overall so far are announced in ascending order.  See Table 2.2.  We call $\psi$ the \emph{rank $R$-tuple of $\pi$}.  As well as being $R$-increasing, it can be seen that $\psi$ is upper:  So $\psi \in UI_R(n)$.

\vspace{.125in}

\begin{figure}[h!]
\begin{center}
\begin{tabular}{lccc}
\underline{Name} & \underline{From/To} & \underline{Input} & \underline{Image} \\ \\
Rank $R$-tuple & $\Psi_R:  S_n^R \rightarrow  UI_R(n)$ & $(2,4,6;1,5,7,8,9;3)$   & $(2,4,6;5,6,7,8,9;9)$ \\ \\
Undoes $\Psi_R|_{S_n^{R\text{-}312}}$ & $\Pi_R:  UG_R(n) \rightarrow S_n^{R\text{-}312}$ & $(2,4,6;4,5,6,7,9;9)$  & $(2,4,6;1,3,5,7,9;8)$ \\ \\
\end{tabular}
\caption*{Table 2.2.  Examples for maps of $R$-tuples for $n = 9$ and $R = \{3, 8 \}$.}
\end{center}
\end{figure}

The map $\Psi_R$ is not injective; for example it maps another $R$-permutation $(2,4,6;3,5,7,8,9;1)$ to the same image as in Table 2.2.  In Proposition \ref{prop320.2}(ii) it will be seen that the restriction of $\Psi_R$ to $S_n^{R\text{-}312}$ is a bijection to $UG_R(n)$ whose inverse is the following map $\Pi_R$:  Let $\gamma \in UG_R(n)$.  See Table 2.2.  Define an $R$-tuple $\Pi_R(\gamma) =: \pi$ by:  Initialize $\pi_i := \gamma_i$ for $i \in (0,q_1]$.  Let $h \in [r]$.  If $\gamma_{q_h} > \gamma_{q_h+1}$, set $s:= \gamma_{q_h} - \gamma_{q_h+1} + 1$.  Otherwise set $s := 0$.  For $i$ in the right side $(q_h + s, q_{h+1}]$ of the $(h+1)^{st}$ carrel, set $\pi_i := \gamma_i$.  For $i$ in the left side $(q_h, q_h + s]$, set $d := q_h + s - i + 1$ and $\pi_i := rank^d( \hspace{1mm} [\gamma_{q_h}] \hspace{1mm} \backslash \hspace{1mm} \{ \pi_1, ... , \pi_{q_h} \} \hspace{1mm} )$.  In words:  working from right to left, fill in the left side by finding the largest element of $[\gamma_{q_h}]$ not used by $\pi$ so far, then the next largest, and so on.  In Table 2.2 when $h=1$ the elements $5,3,1$ are found and placed into the $6^{th}, 5^{th}$, and $4^{th}$ positions.  (Since $\gamma$ is a gapless $R$-tuple, when $s \geq 1$ we have $\gamma_{q_h + s} = \gamma_{q_h}$.  Since `gapless' includes the upper property, here we have $\gamma_{q_h +s} \geq q_h + s$.  Hence $| \hspace{1mm} [\gamma_{q_h}] \hspace{1mm} \backslash \hspace{1mm} \{ \pi_1, ... , \pi_{q_h} \} \hspace{1mm} | \geq s$, and so there are enough elements available to define these left side $\pi_i$. )  Since $\gamma_{q_h} \leq \gamma_{q_{h+1}}$, it can inductively be seen that $\max\{ \pi_1, ... , \pi_{q_h} \} = \gamma_{q_h}$.

When we restrict our attention to the full $R = [n-1]$ case, we will suppress most prefixes and subscripts of `$R$'.  Two examples of this are:  an $[n-1]$-chain becomes a \emph{chain}, and one has $UF(n) = UG(n)$.

\section{Shapes, tableaux, connections to Lie theory}

A \emph{partition} is an $n$-tuple $\lambda \in \mathbb{Z}^n$ such that $\lambda_1 \geq \ldots \geq \lambda_n \geq 0$.  Fix such a $\lambda$ for the rest of the paper.  We say it is \textit{strict} if $\lambda_1 > \ldots > \lambda_n$.  The \textit{shape} of $\lambda$, also denoted $\lambda$, consists of $n$ left justified rows with $\lambda_1, \ldots, \lambda_n$ boxes.  We denote its column lengths by $\zeta_1 \geq \ldots \geq \zeta_{\lambda_1}$.  The column length $n$ is called the \emph{trivial} column length.  Since the columns are more important than the rows, the boxes of $\lambda$ are transpose-indexed by pairs $(j,i)$ such that $1 \leq j \leq \lambda_1$ and $1 \leq i \leq \zeta_j$.  Sometimes for boundary purposes we refer to a $0^{th}$ \emph{latent column} of boxes, which is a prepended $0^{th}$ column of trivial length.  If $\lambda = 0$, its shape is the \textit{empty shape} $\emptyset$.  Define $R_\lambda \subseteq [n-1]$ to be the set of distinct non-trivial column lengths of $\lambda$.  Note that $\lambda$ is strict if and only if $R_\lambda = [n-1]$, i.e. $R_\lambda$ is full.  Set $|\lambda| := \lambda_1 + \ldots + \lambda_n$.

A \textit{(semistandard) tableau of shape $\lambda$} is a filling of $\lambda$ with values from $[n]$ that strictly increase from north to south and weakly increase from west to east.  The example tableau below for $n = 12$ has shape $\lambda = (7^5, 5^4, 2^2)$.  Here $R_\lambda = \{5, 9, 11 \}$.  Let $\mathcal{T}_\lambda$ denote the set of tableaux $T$ of shape $\lambda$.  Under entrywise comparison $\leq$, this set $\mathcal{T}_\lambda$ becomes a poset that is the distributive lattice $L(\lambda, n)$ introduced by Stanley.  The principal ideals $[T]$ in $\mathcal{T}_\lambda$ are clearly convex polytopes in $\mathbb{Z}^{|\lambda|}$.  Fix $T \in \mathcal{T}_\lambda$.  For $j \in [\lambda_1]$, we denote the one column ``subtableau'' on the boxes in the $j^{th}$ column by $T_j$.  Here for $i \in [\zeta_j]$ the tableau value in the $i^{th}$ row is denoted $T_j(i)$.  The set of values in $T_j$ is denoted $B(T_j)$.  Columns $T_j$ of trivial length must be \emph{inert}, that is $B(T_j) = [n]$.  The $0^{th}$ \textit{latent column} $T_0$ is an inert column that is sometimes implicitly prepended to the tableau $T$ at hand:  We ask readers to refer to its values as needed to fulfill definitions or to finish constructions.  We say a tableau $Y$ of shape $\lambda$ is a $\lambda$-\textit{key} if $B(Y_l) \supseteq B(Y_j)$ for $1 \leq l \leq j \leq \lambda_1$.  The example tableau below is a $\lambda$-key.  The empty shape has one tableau on it, the \textit{null tableau}.  Fix a set $Q \subseteq [n]$ with $|Q| =: q \geq 0$.  The \textit{column} $Y(Q)$ is the tableau on the shape for the partition $(1^q, 0^{n-q})$ whose values form the set $Q$.  Then for $d \in [q]$, the value in the $(q+1-d)^{th}$ row of $Y(Q)$ is $rank^d(Q)$.

\begin{figure}[h!]
\begin{center}
\ytableausetup{boxsize = 1.5em}
$$
\begin{ytableau}
1 & 1 & 1 & 1 & 1 & 1 & 1\\
2 & 2 & 3 & 3 & 3 & 4 & 4\\
3 & 3 & 4 & 4 & 4 & 6 & 6\\
4 & 4 & 5 & 5 & 5 & 7 & 7\\
5 & 5 & 6 & 6 & 6 & 10 & 10\\
6 & 6 & 7 & 7 & 7\\
7 & 7 & 8 & 8 & 8\\
8 & 8 & 9 & 9 & 9\\
9 & 9 & 10 & 10 & 10\\
10 & 10 \\
12 & 12
\end{ytableau}$$
\end{center}
\end{figure}

The most important values in a tableau of shape $\lambda$ occur at the ends of its rows.  Using the latent column when needed, these $n$ values from $[n]$ are gathered into an $R_\lambda$-tuple as follows:  Let $T \in \mathcal{T}_\lambda$.  We define the \textit{$\lambda$-row end list} $\omega$ of $T$ to be the $R_\lambda$-tuple given by $\omega_i := T_{\lambda_i}(i)$ for $i \in [n]$.  Note that for $h \in [r+1]$, down the $h^{th}$ ``cliff'' from the right in the shape of $\lambda$ one has $\lambda_i = \lambda_{i^\prime}$ for $i, i^\prime \in (q_{h-1}, q_{h} ]$.  In the example take $h = 2$.  Then $q_2 = 9$ and $q_1 = 5$.  Here $\lambda_i = 5 = \lambda_{i'}$ for $i, i' \in (5,9]$.  Reading off the values of $T$ down that cliff produces the $h^{th}$ cohort of $\omega$.  Here this cohort of $\omega$ is $\{7, 8, 9, 10\}$.  These values are increasing.  So $\omega \in UI_{R_\lambda}(n)$.

For $h \in [r]$, the columns of length $q_h$ in the shape $\lambda$ have indices $j$ such that $j \in (\lambda_{q_{h+1}}, \lambda_{q_h}]$.  When $h=2$ we have $j \in (\lambda_{11}, \lambda_9] = (2,5]$ for columns of length $q_2 = 9$.  A bijection from $R$-chains $B$ to $\lambda$-keys $Y$ is obtained by juxtaposing from left to right $\lambda_n$ inert columns and $\lambda_{q_h}-\lambda_{q_{h+1}}$ copies of $Y(B_h)$ for $r \geq h \geq 1$.   We indicate it by $B \mapsto Y$.  For $h = 2$ here there are $\lambda_9 - \lambda_{11} = 5-2 = 3$ copies of $Y(B_2)$ with $B_2 = (0, 10] \backslash \{2\}$.  Unfortunately we need to have the indices $h$ of the column lengths $q_h$ decreasing from west to east while the column indices $j$ increase from west to east.  Hence the elements of $B_{h+1}$ form the column $Y_j$ for $j = \lambda_{q_{h+1}}$ while the elements of $B_h$ form $Y_{j+1}$.  A bijection from $R_\lambda$-permutations $\pi$ to $\lambda$-keys $Y$ is obtained by following $\pi \mapsto B$ with $B \mapsto Y$.  The image of an $R_\lambda$-permutation $\pi$ is called the \emph{$\lambda$-key of $\pi$}; it is denoted $Y_\lambda(\pi)$.  The example tableau is the $\lambda$-key of $\pi = (1,4,6,7,10; 3,5,8,9; 2, 12; 11)$.  It is easy to see that the $\lambda$-row end list of the $\lambda$-key of $\pi$ is the rank $R_\lambda$-tuple $\Psi_{R_\lambda}(\pi) =: \psi$ of $\pi$:  Here $\psi_i = Y_{\lambda_i}(i)$ for $i \in [n]$.

Let $\alpha \in UI_{R_\lambda}(n)$.  Define $\mathcal{Z}_\lambda(\alpha)$ to be the subset of tableaux $T \in \mathcal{T}_\lambda$ with $\lambda$-row end list $\alpha$.  To see that $\mathcal{Z}_\lambda(\alpha) \neq \emptyset$, for $i \in [n]$ take $T_j(i) := i$ for $j \in [1, \lambda_i)$ and $T_{\lambda_i}(i) := \alpha_i$.  This subset is closed under the join operation for the lattice $\mathcal{T}_\lambda$.  We define the \emph{$\lambda$-row end max tableau $M_\lambda(\alpha)$ for $\alpha$} to be the unique maximal element of $\mathcal{Z}_\lambda(\alpha)$.  The example tableau is an $M_\lambda(\alpha)$.

When we are considering tableaux of shape $\lambda$, much of the data used will be in the form of $R_\lambda$-tuples.  Many of the notions used will be definitions from Section 2 that are being applied with $R := R_\lambda$.  The structure of each proof will depend only upon $R_\lambda$ and not upon how many times a column length is repeated:  If $\lambda^\prime, \lambda^{\prime\prime} \in \Lambda_n^+$ are such that $R_{\lambda^\prime} = R_{\lambda^{\prime\prime}}$, then the development for $\lambda^{\prime\prime}$ will in essence be the same as for $\lambda^\prime$.  To emphasize the original independent entity $\lambda$ and to reduce clutter, from now on rather than writing `$R$' or `$R_\lambda$' we will replace `$R$' by `$\lambda$' in subscripts and in prefixes.  Above we would have written $\omega \in UI_\lambda(n)$ instead of having written $\omega \in UI_{R_\lambda}(n)$ (and instead of having written $\omega \in UI_R(n)$ after setting $R := R_\lambda$).  When $\lambda$ is a strict partition, we omit most `$\lambda$-' prefixes and subscripts since $R_\lambda = [n-1]$.

To connect to Lie theory, fix $R \subseteq [n-1]$ and set $J := [n-1] \backslash R$.  The $R$-permutations are the one-rowed forms of the ``inverses'' of the minimum length representatives collected in $W^J$ for the cosets in $W /W_J$, where $W$ is the Weyl group of type $A_{n-1}$ and $W_J$ is its parabolic subgroup $\langle s_i: i \in J \rangle$.  A partition $\lambda$ is strict exactly when the weight it depicts for $GL(n)$ is strongly dominant.  If we take the set $R$ to be $R_\lambda$, then the restriction of the partial order $\leq$ on $\mathcal{T}_\lambda$ to the $\lambda$-keys depicts the Bruhat order on that $W^J$.  Further details appear in Sections 2, 3, and the appendix of \cite{PW1}.

\section{Rightmost clump deleting chains, gapless $\mathbf{\emph{R}}$-tuples}

We show that if the domain of the simple-minded global bijection $\pi \mapsto B$ is restricted to $S_n^{R\text{-}312} \subseteq S_n^R$, then a bijection to a certain set of chains results.  And while it appears to be difficult to characterize the image $\Psi_R(S_n^R) \subseteq UI_R(n)$ of the $R$-rank map for general $R$, we show that restricting $\Psi_R$ to $S_n^{R\text{-}312}$ produces a bijection to the set $UG_R(n)$ of gapless $R$-tuples.

Given a set of integers, a \emph{clump} of it is a maximal subset of consecutive integers.  After decomposing a set into its clumps, we index the clumps in the increasing order of their elements.  For example, the set $\{ 2,3,5,6,7,10,13,14 \}$ is the union $L_1 \hspace{.5mm} \cup  \hspace{.5mm} L_2  \hspace{.5mm} \cup \hspace{.5mm}  L_3  \hspace{.5mm} \cup  \hspace{.5mm} L_4$,  where  $L_1 := \{ 2,3 \},  L_2 := \{ 5,6,7 \},$ $L_3 := \{ 10 \},  L_4 := \{ 13,14 \}$.

For the first part of this section we temporarily work in the context of the full $R = [n-1]$ case.  A chain  $B$  is \textit{rightmost clump deleting} if for $h \in [n-1]$ the element deleted from each $B_{h+1}$ to produce $B_h$ is chosen from the rightmost clump of $B_{h+1}$.  More formally:   It is rightmost clump deleting if for  $h \in [n-1]$  one has $B_{h} = B_{h+1} \backslash \{ b \}$  only when  $[b, m] \subseteq B_{h+1}$,  where  $m := max (B_{h+1})$.  For $n = 3$ there are five rightmost clump deleting chains, whose sets $B_3 \supset B_2 \supset B_1$ are displayed from the top in three rows:

\begin{figure}[h!]
\begin{center}
\setlength\tabcolsep{.1cm}
\begin{tabular}{ccccc}
1& &2& &\cancel{3}\\
 &1& &\cancel{2}& \\
 & &\cancel{1}& & \\
\end{tabular}\hspace{10mm}
\begin{tabular}{ccccc}
1& &2& &\cancel{3}\\
 &\cancel{1}& &2& \\
 & &\cancel{2}& & \\
\end{tabular}\hspace{10mm}
\begin{tabular}{ccccc}
1& &\cancel{2}& &3\\
 &1& &\cancel{3}& \\
 & &\cancel{1}& & \\
\end{tabular}\hspace{10mm}
\begin{tabular}{ccccc}
\cancel{1}& &2& &3\\
 &2& &\cancel{3}& \\
 & &\cancel{2}& & \\
\end{tabular}\hspace{10mm}
\begin{tabular}{ccccc}
\cancel{1}& &2& &3\\
 &\cancel{2}& &3& \\
 & &\cancel{3}& & \\
\end{tabular}
\end{center}
\end{figure}

\noindent To form the corresponding $\pi$, record the deleted elements from bottom to top.  Note that the 312-containing permutation $(3;1;2)$ does not occur.  Its triangular display of $B_3 \supset B_2 \supset B_1$ deletes the `1' from the ``left'' clump in the second row.

After Part (0) restates the definition of this concept, we present four reformulations of it:

\begin{fact}\label{fact320.1}Let $B$ be a chain.  Set $\{ b_{h+1} \} := B_{h+1} \backslash B_h$ for $h \in [n-1]$.  Set $m_h := \max (B_h)$ for $h \in [n]$.  The following conditions are equivalent to this chain being rightmost clump deleting:

\noindent(0)  For $h \in [n-1]$, one has $[b_{h+1}, m_{h+1}] \subseteq B_{h+1}$.

\noindent(i)  For $h \in [n-1]$, one has $[b_{h+1}, m_h] \subseteq B_{h+1}$.

\noindent(ii)  For $h \in [n-1]$, one has $(b_{h+1}, m_h) \subset B_h$.

\noindent(iii)  For $h \in [n-1]$:  If $b_{h+1} < m_h$, then $b_{h+1} = \max([m_h] \backslash B_h)$.

\noindent(iii$^\prime$)  For $h \in [n-1]$, one has $b_{h+1} = \max([m_{h+1}] \backslash B_h)$. \end{fact}

The following characterization is related to Part (ii) of the preceeding fact via the correspondence $\pi \longleftrightarrow B$:

\begin{fact}\label{fact320.2}A permutation $\pi$ is 312-avoiding if and only if for every $h \in [n-1]$ we have \\ $(\pi_{h+1}, \max\{\pi_1, ... , \pi_{h}\}) \subset \{ \pi_1, ... , \pi_{h} \}$. \end{fact}

Since the following result will be generalized by Proposition \ref{prop320.2}, we do not prove it here.  Part (i) is Exercise 2.202 of \cite{Sta}.

\begin{prop}\label{prop320.1}For the full $R = [n-1]$ case we have:

\noindent (i)  The restriction of the global bijection $\pi \mapsto B$ from $S_n$ to $S_n^{312}$ is a bijection to the set of rightmost clump deleting chains.  Hence there are $C_n$ rightmost clump deleting chains.

\noindent (ii)  The restriction of the rank tuple map $\Psi$ from $S_n$ to $S_n^{312}$ is a bijection to $UF(n)$ whose inverse is $\Pi$.  \end{prop}

\noindent Here when $R = [n-1]$, the map $\Pi : UF(n) \longrightarrow S_n^{312}$ has a simple description.  It was introduced in \cite{PS} for Theorem 14.1.  Given an upper flag $\phi$, recursively construct $\Pi(\phi) =: \pi$ as follows:  Start with $\pi_1 := \phi_1$.  For $i \in [n-1]$, choose $\pi_{i+1}$ to be the maximum element of $[\phi_{i+1}] \backslash \{ \pi_1, ... , \pi_{i} \}$.

Now fix $R \subseteq [n-1]$.  Let $B$ be an $R$-chain.  More generally, we say $B$ is \textit{$R$-rightmost clump deleting} if this condition holds for each $h \in [r]$:  Let $B_{h+1} =: L_1 \cup L_2 \cup ... \cup L_f$ decompose $B_{h+1}$ into clumps for some $f \geq 1$.  We require $L_e \cup L_{e+1} \cup ... \cup L_f \supseteq B_{h+1} \backslash B_{h} \supseteq L_{e+1} \cup ... \cup L_f$ for some $e \in [f]$.  This condition requires the set $B_{h+1} \backslash B_h$ of new elements that augment the set $B_h$ of old elements to consist of entirely new clumps $L_{e+1}, L_{e+2}, ... , L_f$, plus some further new elements that combine with some old elements to form the next ``lower'' clump $L_e$ in $B_{h+1}$.  When $n = 14$ and $R = \{3, 5, 10 \}$, an example of an $R$-rightmost clump deleting chain is given by $\emptyset \subset \{ 1 \text{-} 2, 6 \} \subset \{ 1\text{-}2, 5\text{-}6, 8 \} \subset \{ 1\text{-}2, 4\text{-}5\text{-}6\text{-}7\text{-}8, 10, 13\text{-}14 \}$ $\subset \{1\text{-}2\text{-}3\text{-}...\text{-}13\text{-}14\}$.  Here are some reformulations of the notion of $R$-rightmost clump deleting:

\begin{fact}\label{fact320.3}Let $B$ be an $R$-chain.  For $h \in [r]$, set $b_{h+1} := \min (B_{h+1} \backslash B_{h} )$ and $m_h := \max (B_h)$.  The following conditions are equivalent to this chain being $R$-rightmost clump deleting:

\noindent (i)  For $h \in [r]$, one has $[b_{h+1}, m_{h}] \subseteq B_{h+1}$.

\noindent (ii)  For $h \in [r]$, one has $(b_{h+1}, m_{h}) \subset B_{h+1}$.

\noindent (iii)  For $h \in [r]$, let $s$ be the number of elements of $B_{h+1} \backslash B_{h}$ that are less than $m_{h}$.  These must be the $s$ largest elements of $[m_{h}] \backslash B_{h}$.  \end{fact}

\noindent Part (iii) will again be used in \cite{PW3} for projecting and lifting 312-avoidance.

The following characterization is related to Part (ii) of the preceding fact via the correspondence $\pi \longleftrightarrow B$:

\begin{fact}\label{fact320.4}An $R$-permutation $\pi$ is $R$-312-avoiding if and only if for every $h \in [r]$ one has \\ $( \min\{\pi_{q_{h}+1}, ... , \pi_{q_{h+1}} \} , \max \{\pi_1, ... , \pi_{q_{h}}\} ) \subset \{ \pi_1, ... , \pi_{q_{h+1}} \}$.  \end{fact}

Is it possible to characterize the rank $R$-tuple $\Psi_R(\pi) =: \psi$ of an arbitrary $R$-permutation $\pi$?  An \emph{$R$-flag} is an $R$-increasing upper tuple $\varepsilon$ such that $\varepsilon_{q_{h+1} +1 - u} \geq \varepsilon_{q_{h} +1 - u}$ for $h \in [r]$ and $u \in [\min\{ p_{h+1}, p_{h}\}]$.  It can be seen that $\psi$ is necessarily an $R$-flag.  But the three conditions required so far (upper, $R$-increasing, $R$-flag) are not sufficient:  When $n = 4$ and $R = \{ 1, 3 \}$, the $R$-flag $(3;2,4;4)$ cannot arise as the rank $R$-tuple of an $R$-permutation.  In contrast to the upper flag characterization in the full case, it might not be possible to develop a simply stated sufficient condition for an $R$-tuple to be the rank $R$-tuple  $\Psi_R(\pi)$ of a general $R$-permutation $\pi$.  But it can be seen that the rank $R$-tuple $\psi$ of an $R$-312-avoiding permutation $\pi$ is necessarily a gapless $R$-tuple, since a failure of `gapless' for $\psi$ leads to the containment of an $R$-312 pattern.  Building upon the observation that $UG(n) = UF(n)$ in the full case, this seems to indicate that the notion of ``gapless $R$-tuple'' is the correct generalization of the notion of ``flag'' from $[n-1]$-tuples to $R$-tuples.  (It can be seen directly that a gapless $R$-tuple is necessarily an $R$-flag.)

Two bijections lie at the heart of this work; the second one will again be used in \cite{PW3} to prove Theorem 9.1.

\begin{prop}\label{prop320.2}For general $R \subseteq [n-1]$ we have:

\noindent (i)  The restriction of the global bijection $\pi \mapsto B$ from $S_n^R$ to $S_n^{R\text{-}312}$ is a bijection to the set of $R$-rightmost clump deleting chains.

\noindent (ii)  The restriction of the rank $R$-tuple map $\Psi_R$ from $S_n^R$ to $S_n^{R\text{-}312}$ is a bijection to $UG_R(n)$ whose inverse is $\Pi_R$.  \end{prop}

\begin{proof}Setting $b_h = \min\{\pi_{q_{h}+1}, ... , \pi_{q_{h+1}} \}$ and $m_{h} = \max \{\pi_1, ... , \pi_{q_{h}}\}$, use Fact \ref{fact320.4}, the $\pi \mapsto B$ bijection, and Fact \ref{fact320.3}(ii) to confirm (i).  As noted above, the restriction of $\Psi_R$ to $S_n^{R\text{-}312}$ gives a map to $UG_R(n)$.  Let $\gamma \in UG_R(n)$ and construct $\Pi_R(\gamma) =: \pi$.  Let $h \in [r]$.  Recall that $\max\{ \pi_1, ... , \pi_{q_h} \} = \gamma_{q_h}$.  Since $\gamma$ is $R$-increasing it can be seen that the $\pi_i$ are distinct.  So $\pi$ is an $R$-permutation.  Let $s \geq 0$ be the number of entries of $\{ \pi_{q_{h}+1} , ... , \pi_{q_{h+1}} \}$ that are less than $\gamma_{q_{h}}$.  These are the $s$ largest elements of $[\gamma_{q_{h}}] \backslash \{ \pi_1, ... , \pi_{q_{h}} \}$.  If in the hypothesis of Fact \ref{fact320.3} we take $B_h := \{\pi_1, ... , \pi_{q_h} \}$, we have $m_h = \gamma_{q_h}$.  So the chain $B$ corresponding to $\pi$ satisfies Fact \ref{fact320.3}(iii).  Since Fact \ref{fact320.3}(ii) is the same as the characterization of an $R$-312-avoiding permutation in Fact \ref{fact320.4}, we see that $\pi$ is $R$-312-avoiding.  It can be seen that $\Psi_R[\Pi_R(\gamma)] = \gamma$, and so $\Psi_R$ is surjective from $S_n^{R\text{-}312}$ to $UG_R(n)$.  For the injectivity of $\Psi_R$, now let $\pi$ denote an arbitrary $R$-312-avoiding permutation.  Form $\Psi_R(\pi)$, which is a gapless $R$-tuple.  Using Facts \ref{fact320.4} and \ref{fact320.3}, it can be seen that $\Pi_R[\Psi_R(\pi)] = \pi$.  Hence $\Psi_R$ is injective.  \end{proof}

\section{Row end max tableaux, gapless (312-avoiding) keys}

We study the $\lambda$-row end max tableaux of gapless $\lambda$-tuples.  We also form the $\lambda$-keys of the $R$-312-avoiding permutations and introduce ``gapless'' lambda-keys.  We show that these three sets of tableaux coincide.

Let $\alpha \in UI_\lambda(n)$.  The values of the $\lambda$-row end max tableau $M_{\lambda}(\alpha) =: M$ can be determined as follows:  For $h \in [r]$ and $j \in (\lambda_{q_{h+1}}, \lambda_{q_h}]$, first set $M_j(i) = \alpha_i$ for $i \in (q_{h-1}, q_h]$.  When $h > 1$, from east to west among columns and south to north within a column, also set $M_j(i) := \min\{ M_j(i+1)-1, M_{j+1}(i) \}$ for $i \in (0, q_{h-1}]$.  Finally, set $M_j(i) := i$ for $j \in (0, \lambda_n]$ and $i \in (0,n]$.  (When $\zeta_j = \zeta_{j+1}$, this process yields $M_j = M_{j+1}$.)  The example tableau in Section 3 is $M_\lambda(\alpha)$ for $\alpha = (1,4,6,7,10;7,8,9,10;10,12; 12)$.  There we have $s = 4$ and $s = 1$ respectively for $h =1$ and $h=2$:

\begin{lem}\label{lemma340.1}Let $\gamma$ be a gapless $\lambda$-tuple.  The $\lambda$-row end max tableau $M_{\lambda}(\gamma) =: M$ is a key.  For $h \in [r]$ and $j := \lambda_{q_{h+1}}$, the $s \geq 0$ elements in $B(M_{j}) \backslash B(M_{j+1})$ that are less than $M_{j+1}(q_{h}) = \gamma_{q_{h}}$ are the $s$ largest elements of $[\gamma_{q_{h}}] \backslash B(M_{j+1})$. \end{lem}

\begin{proof} Let $h \in [r]$ and set $j := \lambda_{q_{h+1}}$.  We claim $B(M_{j+1}) \subseteq B(M_j)$.  If $M_j(q_{h} + 1) = \gamma_{q_{h}+1}  >  \gamma_{q_h} = M_{j+1}(q_{h})$, then $M_j(i) = M_{j+1}(i)$ for $i \in (0, q_{h}]$ and the claim holds.  Otherwise $\gamma_{q_{h} + 1} \leq \gamma_{q_{h}}$.  The gapless condition on $\gamma$ implies that if we start at $(j, q_{h}+1)$ and move south, the successive values in $M_j$ increment by 1 until some lower box has the value $\gamma_{q_{h}}$.  Let $i \in (q_{h}, q_{h+1}]$ be the index such that $M_j(i) = \gamma_{q_{h}}$.  Now moving north from $(j,i)$, the values in $M_j$ decrement by 1 either all of the way to the top of the column, or until there is a row index $k \in (0, q_{h})$ such that $M_{j+1}(k) < M_j(k+1)-1$.  In the former case set $k := 0$.  In the example we have $k=1$ and $k=0$ respectively for $h=1$ and $h=2$.  If $k > 0$ we have $M_j(x) = M_{j+1}(x)$ for $x \in (0,k]$.  Now use  $M_j(k+1) \leq M_{j+1}( k+1)$ to see that the values $M_{j + 1}(k+1), M_{j+1}( k+2), ... , M_{j+1}( q_{h})$ each appear in the interval of values $[ M_j(k+1), M_j(i) ]$.  Thus $B(M_{j+1}) \subseteq B(M_j)$.  Using the parenthetical remark made before the lemma's statement, we see that $M$ is a key.  There are $q_{h+1} - i$ elements in $B(M_j) \backslash B(M_{j+1})$ that are larger than $M_{j+1}( q_{h}) = \gamma_{q_{h}}$.  So $s := (q_{h+1} - q_{h}) - (q_{h+1} - i) \geq 0$ is the number of values in $B(M_j) \backslash B(M_{j+1})$ that are less than $\gamma_{q_{h}}$.  These $s$ values are the complement in $[ M_j(k+1), M_j(i) ]$ of the set $\{ \hspace{1mm} M_{j+1}(x) : x \in [k+1, q_{h}] \hspace{1mm} \}$, where $M_j(i) = M_{j+1}(q_{h}) = \gamma_{q_{h}}$.  \end{proof}

We now introduce a tableau analog to the notion of ``$R$-rightmost clump deleting chain''.  A $\lambda$-key $Y$ is \textit{gapless} if the condition below is satisfied for $h \in [r-1]$:  Let $b$ be the smallest value in a column of length $q_{h+1}$ that does not appear in a column of length $q_{h}$.  For $j \in (\lambda_{q_{h +2}}, \lambda_{q_{h+1}}]$, let $i \in (0, q_{h+1}]$ be the shared row index for the occurrences of $b = Y_j(i)$.  Let $m$ be the bottom (largest) value in the columns of length $q_{h}$.  If $b > m$ there are no requirements.  Otherwise:  For $j \in (\lambda_{q_{h +2}}, \lambda_{q_{h+1}}]$, let $k \in (i, q_{h+1}]$ be the shared row index for the occurrences of $m = Y_j(k)$.  For $j \in (\lambda_{q_{h + 2}}, \lambda_{q_{h+1}}]$ one must have $Y_j(i+1) = b+1, Y_j(i+2) = b+2, ... , Y_j(k-1) = m-1$ holding between $Y_j(i) = b$ and $Y_j(k) = m$.  (Hence necessarily $m - b = k - i$.)  The tableau shown above is a gapless $\lambda$-key.

Given a partition $\lambda$ with $R_\lambda =: R$, our next result considers three sets of $R$-tuples and three sets of tableaux of shape $\lambda$:

\noindent (a)   The set  $\mathcal{A}_R$  of $R$-312-avoiding permutations and the set $\mathcal{P}_\lambda$  of their $\lambda$-keys.

\noindent (b)   The set  $\mathcal{B}_R$  of $R$-rightmost clump deleting chains and the set  $\mathcal{Q}_\lambda$ of gapless $\lambda$-keys.

\noindent (c)   The set  $\mathcal{C}_R$  of gapless $R$-tuples and the set  $\mathcal{R}_\lambda$  of their $\lambda$-row end max tableaux.

\newpage

\begin{thm}\label{theorem340}Let $\lambda$ be a partition and set $R := R_\lambda$.

\noindent (i) The three sets of tableaux coincide:  $\mathcal{P}_\lambda = \mathcal{Q}_\lambda = \mathcal{R}_\lambda$.

\noindent (ii)  An $R$-permutation is  $R$-312-avoiding if and only if its $\lambda$-key is gapless.

\noindent (iii) If an $R$-permutation is $R$-312-avoiding, then the $\lambda$-row end max tableau of its rank $R$-tuple is its $\lambda$-key.

\noindent The restriction of the global bijection $B \mapsto Y$ from all $R$-chains to $R$-rightmost clump deleting chains is a bijection from $\mathcal{B}_R$ to $\mathcal{Q}_\lambda$.  The process of constructing the $\lambda$-row end max tableau is a bijection from $\mathcal{C}_R$ to $\mathcal{R}_\lambda$.  The bijection $\pi \mapsto B$ from $\mathcal{A}_R$ to $\mathcal{B}_R$ induces a map from $\mathcal{P}_\lambda$ to $\mathcal{Q}_\lambda$ that is the identity.  The bijection $\Psi_R$ from $\mathcal{A}_R$ to $\mathcal{C}_R$ induces a map from $\mathcal{P}_\lambda$ to $\mathcal{R}_\lambda$ that is the identity.\end{thm}

\noindent Part (iii) will again be used in \cite{PW3} to prove Theorem 9.1; there it will also be needed for the discussion in Section 12.  In the full case when $\lambda$ is strict and $R = [n-1]$, the converse of Part (iii) holds:  If the row end max tableau of the rank tuple of a permutation is the key of the permutation, then the permutation is 312-avoiding.  For a counterexample to this converse for general $\lambda$, choose $n = 4, \lambda = (2,1,1,0)$, and $\pi = (4;1,2;3)$.  Then $Y_\lambda(\pi) = M_\lambda(\psi)$ with $\pi \notin S_n^{R\text{-}312}$.  The bijection from $\mathcal{C}_R$ to $\mathcal{R}_\lambda$ and the equality $\mathcal{Q}_\lambda = \mathcal{R}_\lambda$ imply that an $R$-tuple is $R$-gapless if and only if it arises as the $\lambda$-row end list of a gapless $\lambda$-key.

\begin{proof}

\noindent For the first of the four map statements, use the $B \mapsto Y$ bijection to relate Fact \ref{fact320.3}(i) to the definition of gapless $\lambda$-key.  The map in the second map statement is surjective by definition and is also obviously injective.  Use the construction of the bijection $\pi \mapsto B$ and the first map statement to confirm the equality $\mathcal{P}_\lambda = \mathcal{Q}_\lambda$ and the third map statement.  Part (ii) follows.

To prove Part (iii), let $\pi \in S_n^{R\text{-}312}$.  Create the $R$-chain $B$ corresponding to $\pi$ and then its $\lambda$-key $Y := Y_\lambda(\pi)$.  Set $\gamma := \Psi_R(\pi)$ and then $M := M_\lambda(\gamma)$.  Clearly $B(Y_{\lambda_v}) = B_1 = \{ \gamma_1, ... , \gamma_v \} = B(M_{\lambda_v})$ for $v := q_1$.  Proceed by induction on $h \in [r]$:  For $v := q_h$ assume $B(Y_{\lambda_v}) = B(M_{\lambda_v})$.  So $\max[B(Y_{\lambda_v})] = Y_{\lambda_v}(v) = M_{\lambda_v}(v) = \gamma_v$.  Rename the example $\alpha$ before Lemma \ref{lemma340.1} as $\gamma$.  Viewing that tableau as $M_\lambda( \gamma ) =: M$, for $h=2$ we have $M_5(9) = \gamma_9 = 10$.  Set $v^\prime := q_{h+1}$.  Let $s_Y$ be the number of values in $B(Y_{\lambda_{v^\prime}}) \backslash B(Y_{\lambda_v})$ that are less than $\gamma_v$.  Viewing the example tableau as $Y$, for $h=2$ we have $s_Y = 1$.  Since $\gamma_v \in B(Y_{\lambda_v})$, the number of values in $B(Y_{\lambda_{v^\prime}}) \backslash B(Y_{\lambda_v})$ that exceed $\gamma_v$ is $p_{h+1} - s_Y$.  These values are the entries in $\{ \pi_{v+1} , ... , \pi_{v^\prime} \}$ that exceed $\gamma_v$.  So from $\gamma := \Psi_R(\pi)$ and the description of $M_\lambda(\gamma)$ it can be seen that these values are exactly the values in $B(M_{\lambda_{v^\prime}}) \backslash B(M_{\lambda_v})$ that exceed $\gamma_v$.  Let $s_M$ be the number of values in $B(M_{\lambda_{v^\prime}}) \backslash B(M_{\lambda_v})$ that are less than $\gamma_v$.  Since $M$ is a key by Lemma \ref{lemma340.1} and $\gamma_v \in B(M_{\lambda_v})$, we have $s_M = p_{h+1} - (p_{h+1}-s_Y) = s_Y =: s$.  From Proposition \ref{prop320.2}(i) we know that $B$ is $R$-rightmost clump deleting.  By Fact \ref{fact320.3}(iii) applied to $B$ and Lemma \ref{lemma340.1} applied to $\gamma$, we see that for both $Y$ and for $M$ the ``new'' values that are less than $\gamma_v$ are the $s$ largest elements of $[\gamma_v] \backslash B(Y_{\lambda_v}) = [\gamma_v] \backslash B(M_{\lambda_v})$.  Hence $Y_{\lambda_{v^\prime}} = M_{\lambda_{v^\prime}}$.  Since we only need to consider the rightmost columns of each length when showing that two $\lambda$-keys are equal, we have $Y = M$.  The equality $\mathcal{P}_\lambda = \mathcal{R}_\lambda$ and the final map statement follow.  \end{proof}

\begin{cor}When $\lambda$ is strict, there are $C_n$ gapless $\lambda$-keys. \end{cor}

\section{Sufficient condition for Demazure convexity}

Fix a $\lambda$-permutation $\pi$.  We define the set $\mathcal{D}_\lambda(\pi)$ of Demazure tableaux.  We show that if $\pi$ is $\lambda$-312-avoiding, then the tableau set $\mathcal{D}_\lambda(\pi)$ is the principal ideal $[Y_\lambda(\pi)]$.

First we need to specify how to form the \emph{scanning tableau} $S(T)$ for a given $T \in \mathcal{T}_\lambda$.   See page 394 of \cite{Wi2} for an example of this method.  Given a sequence $x_1, x_2, ...$, its \emph{earliest weakly increasing subsequence (EWIS)} is $x_{i_1}, x_{i_2}, ...$, where $i_1 = 1$ and for $u > 1$ the index $i_u$ is the smallest index satisfying $x_{i_u} \geq x_{i_{u-1}}$.  Let $T \in \mathcal{T}_\lambda$.  Draw the shape $\lambda$ and fill its boxes as follows to produce $S(T)$:  Form the sequence of the bottom values of the columns of $T$ from left to right.  Find the EWIS of this sequence, and mark each box that contributes its value to this EWIS.  The sequence of locations of the marked boxes for a given EWIS is its \emph{scanning path}.  Place the final value of this EWIS in the lowest available location in the leftmost available column of $S(T)$.  This procedure can be repeated as if the marked boxes are no longer part of $T$, since it can be seen that the unmarked locations form the shape of some partition.  Ignoring the marked boxes, repeat this procedure to fill in the next-lower value of $S(T)$ in its first column.  Once all of the scanning paths originating in the first column have been found, every location in $T$ has been marked and the first column of $S(T)$ has been created.  For $j> 1$, to fill in the $j^{th}$ column of $S(T)$:  Ignore the leftmost $(j-1)$ columns of $T$, remove all of the earlier marks from the other columns, and repeat the above procedure.  The scanning path originating at a location $(l,k) \in \lambda$ is denoted $\mathcal{P}(T;l,k)$.  It was shown in \cite{Wi2} that $S(T)$ is the ``right key'' of Lascoux and Sch\"{u}tzenberger for $T$, which was denoted $R(T)$ there.

As in \cite{PW1}, we now use the $\lambda$-key $Y_\lambda(\pi)$ of $\pi$ to define the set of \emph{Demazure tableaux}:  $\mathcal{D}_\lambda(\pi) :=$ \\ $\{ T \in \mathcal{T}_\lambda : S(T) \leq Y_\lambda(\pi) \}$.  We list some basic facts concerning keys, scanning tableaux, and sets of Demazure tableaux.  Since it has long been known that $R(T)$ is a key for any $T \in \mathcal{T}_\lambda$, having $S(T) = R(T)$ gives Part (i).  Part (ii) is easy to deduce from the specification of the scanning method.  The remaining parts follow in succession from Part (ii) and the bijection $\pi \mapsto Y$.

\begin{fact}\label{fact420}Let $T \in \mathcal{T}_\lambda$ and let $Y \in \mathcal{T}_\lambda$ be a key.

\noindent (i)  $S(T)$ is a key and hence $S(T) \in \mathcal{T}_\lambda$.

\noindent (ii)  $T \leq S(T)$ and $S(Y) = Y$.

\noindent (iii)  $Y_\lambda(\pi) \in \mathcal{D}_\lambda(\pi)$ and $\mathcal{D}_\lambda(\pi) \subseteq [Y_\lambda(\pi)]$.

\noindent (iv)  The unique maximal element of $\mathcal{D}_\lambda(\pi)$ is $Y_\lambda(\pi)$.

\noindent (v)  The Demazure sets $\mathcal{D}_\lambda(\sigma)$ of tableaux are nonempty subsets of $\mathcal{T}_\lambda$ that are precisely indexed by the $\sigma \in S_n^\lambda$.  \end{fact}

For $U \in \mathcal{T}_\lambda$, define $m(U)$ to be the maximum value in $U$.  (Define $m(U) := 1$ if $U$ is the null tableau.)  Let $T \in \mathcal{T}_\lambda$.  Let $(l,k) \in \lambda$.  As in Section 4 of \cite{PW1}, define $U^{(l,k)}$ to be the tableau formed from $T$ by finding and removing the scanning paths that begin at $(l,\zeta_l)$ through $(l, k+1)$, and then removing the $1^{st}$ through $l^{th}$ columns of $T$.  (If $l = \lambda_1$, then $U^{(l,k)}$ is the null tableau for any $k \in [\zeta_{\lambda_1}]$.)  Set $S := S(T)$.  Lemma 4.1 of \cite{PW1} states that $S_l(k) = \text{max} \{ T_l(k), m(U^{(l,k)}) \}$.

To reduce clutter in the proofs we write $Y_\lambda(\pi) =: Y$ and $S(T) =: S$.

\begin{prop}\label{prop420.1}Let $\pi \in S^\lambda_n$ and $T \in \mathcal{T}_\lambda$ be such that $T \leq Y_\lambda(\pi)$.  If there exists $(l,k) \in \lambda$ such that $Y_l(k) < m(U^{(l,k)})$, then $\pi$ is $\lambda$-312-containing.  \end{prop}

\begin{proof}Reading the columns from right to left and then each column from bottom to top, let $(l,k)$ be the first location in $\lambda$ such that $m(U^{(l,k)}) > Y_l(k)$.  In the rightmost column we have $m(U^{(\lambda_1,i)}) = 1$ for all $i \in [\zeta_{\lambda_1}]$.  Thus $m(U^{(\lambda_1,i)}) \leq Y_{\lambda_1}(i)$ for all $i \in [\zeta_{\lambda_1}]$.  So we must have $l \in [1, \lambda_1)$.  There exists $j > l$ and $i \leq k$ such that $m(U^{(l,k)}) = T_j(i)$.  Since $T \leq Y$, so far we have $Y_l(k) < T_j(i) \leq Y_j(i)$.  Note that since $Y$ is a key we have $k < \zeta_l$.  Then for $k < f \leq \zeta_l$ we have $m(U^{(l,f)}) \leq Y_l(f)$.  So $T \leq Y$ implies that $S_l(f) \leq Y_l(f)$ for $k < f \leq \zeta_l$.

Assume for the sake of contradiction that $\pi$ is $\lambda$-312-avoiding.  Theorem \ref{theorem340}(ii) says that its $\lambda$-key $Y$ is gapless.  If the value $Y_l(k)$ does not appear in $Y_j$, then the columns that contain $Y_l(k)$ must also contain $[Y_l(k), Y_j(i)]$:  Otherwise, the rightmost column that contains $Y_l(k)$ has index $\lambda_{q_{h+1}}$ for some $h \in [r-1]$ and there exists some $u \in [Y_l(k), Y_j(i)]$ such that $u \notin Y_{\lambda_{q_{h+1}}}$.  Then $Y$ would not satisfy the definition of gapless $\lambda$-key, since for this $h+1$ in that definition one has $b \leq u$ and $u \leq m$.  If the value $Y_l(k)$ does appear in $Y_j$, it appears to the north of $Y_j(i)$ there.  Then $i \leq k$ implies that some value $Y_l(f) < Y_j(i)$ with $f < k$ does not appear in $Y_j$.  As above, the columns that contain the value $Y_l(f) < Y_l(k)$ must also contain $[Y_l(f), Y_j(i)]$.  In either case $Y_l$ must contain $[Y_l(k), Y_j(i)]$.  This includes $T_j(i)$.

Now let $f > k$ be such that $Y_l(f) = T_j(i)$.  Then we have $S_l(f) > S_l(k) = \text{max} \{T_l(k),m(U^{(l,k)}) \}$ $\geq T_j(i) = Y_l(f)$.  This is our desired contradiction.  \end{proof}

As in Section 5 of \cite{PW1}: When $m(U^{(l,k)}) > Y_l(k)$, define the set $A_\lambda(T,\pi;l,k) := \emptyset$.  Otherwise, define $A_\lambda(T,\pi;l,k) := [ k , \text{min} \{ Y_l(k), T_l(k+1) -1, T_{l+1}(k) \} ] $.  (Refer to fictitious bounding values $T_l(\zeta_l + 1) := n+1$ and $T_{\lambda_l + 1}(l) := n$.)

\begin{thm}\label{theorem420}Let $\lambda$ be a partition and $\pi$ be a $\lambda$-permutation.  If $\pi$ is $\lambda$-312-avoiding, then $\mathcal{D}_\lambda(\pi) = [Y_\lambda(\pi)]$.  \end{thm}

\begin{proof}The easy containment $\mathcal{D}_\lambda(\pi) \subseteq [Y_\lambda(\pi)]$ is Fact \ref{fact420}(iii).  Conversely, let $T \leq Y$ and $(l,k) \in \lambda$.  The contrapositive of Proposition \ref{prop420.1} gives $A_\lambda(T,\pi;l,k) = [ k , \text{min} \{ Y_l(k), T_l(k+1) - 1, T_{l+1}(k) \} ]$.  Since $T \leq Y$, we see that $T_l(k) \in A_\lambda(T,\pi;l,k)$ for all $(l,k) \in \lambda$.  Theorem 5.1 of \cite{PW1} says that $T \in \mathcal{D}_\lambda(\pi)$.  \end{proof}

\noindent This result is used in \cite{PW3} to prove Theorem 9.1(ii).

\section{Necessary condition for Demazure convexity}

Continue to fix a $\lambda$-permutation $\pi$.  We show that $\pi$ must be $\lambda$-312-avoiding for the set of Demazure tableaux $\mathcal{D}_\lambda(\pi)$ to be a convex polytope in $\mathbb{Z}^{|\lambda|}$.  We do so by showing that if $\pi$ is $\lambda$-312-containing, then $\mathcal{D}_\lambda(\pi)$ does not contain a particular semistandard tableau that lies on the line segment defined by two particular keys that are in $\mathcal{D}_\lambda(\pi)$.

\begin{thm}\label{theorem520}Let $\lambda$ be a partition and let $\pi$ be a $\lambda$-permutation.  If $\mathcal{D}_\lambda(\pi)$ is convex in $\mathbb{Z}^{|\lambda|}$, then $\pi$ is $\lambda$-312-avoiding.  \end{thm}

\noindent This result is used in \cite{PW3} to prove Theorem 9.1(iii) and Theorem 10.3.

\begin{proof}For the contrapositive, assume that $\pi$ is $\lambda$-312-containing.  Here $r := |R_\lambda| \geq 2$.  There exists $1 \leq g < h \leq r$ and some $a \leq q_g < b \leq q_h < c$ such that $\pi_b < \pi_c < \pi_a$.  Among such patterns, we specify one that is optimal for our purposes.  Figure 7.1 charts the following choices for $\pi = (4,8;9;2,3;1,5;6,7)$ in the first quadrant.  Choose $h$ to be minimal.  So $b \in (q_{h-1}, q_h]$.  Then choose $b$ so that $\pi_b$ is maximal.  Then choose $a$ so that $\pi_a$ is minimal.  Then choose $g$ to be minimal.  So $a \in (q_{g-1} , q_g]$.  Then choose any $c$ so that $\pi_c$ completes the $\lambda$-312-containing condition.

These choices have led to the following two prohibitions; see the rectangular regions in Figure 7.1:

\noindent (i) By the minimality of $h$ and the maximality of $\pi_b$, there does not exist $e \in (q_g, q_h]$ such that $\pi_b < \pi_e < \pi_c$.

\noindent (ii) By the minimality of $\pi_a$, there does not exist $e \in [q_{h-1}]$ such that $\pi_c < \pi_e < \pi_a$.

\noindent If there exists $e \in [q_g]$ such that $\pi_b < \pi_e < \pi_c$, choose $d \in [q_g]$ such that $\pi_d$ is maximal with respect to this condition; otherwise set $d = b$.  So $\pi_b \leq \pi_d$ with $d \leq b$.  We have also ruled out:

\noindent (iii) By the maximality of $\pi_d$, there does not exist $e \in [q_g]$ such that $\pi_d < \pi_e < \pi_c$.

\begin{figure}[h!]
  \begin{center}
    \includegraphics[scale=1]{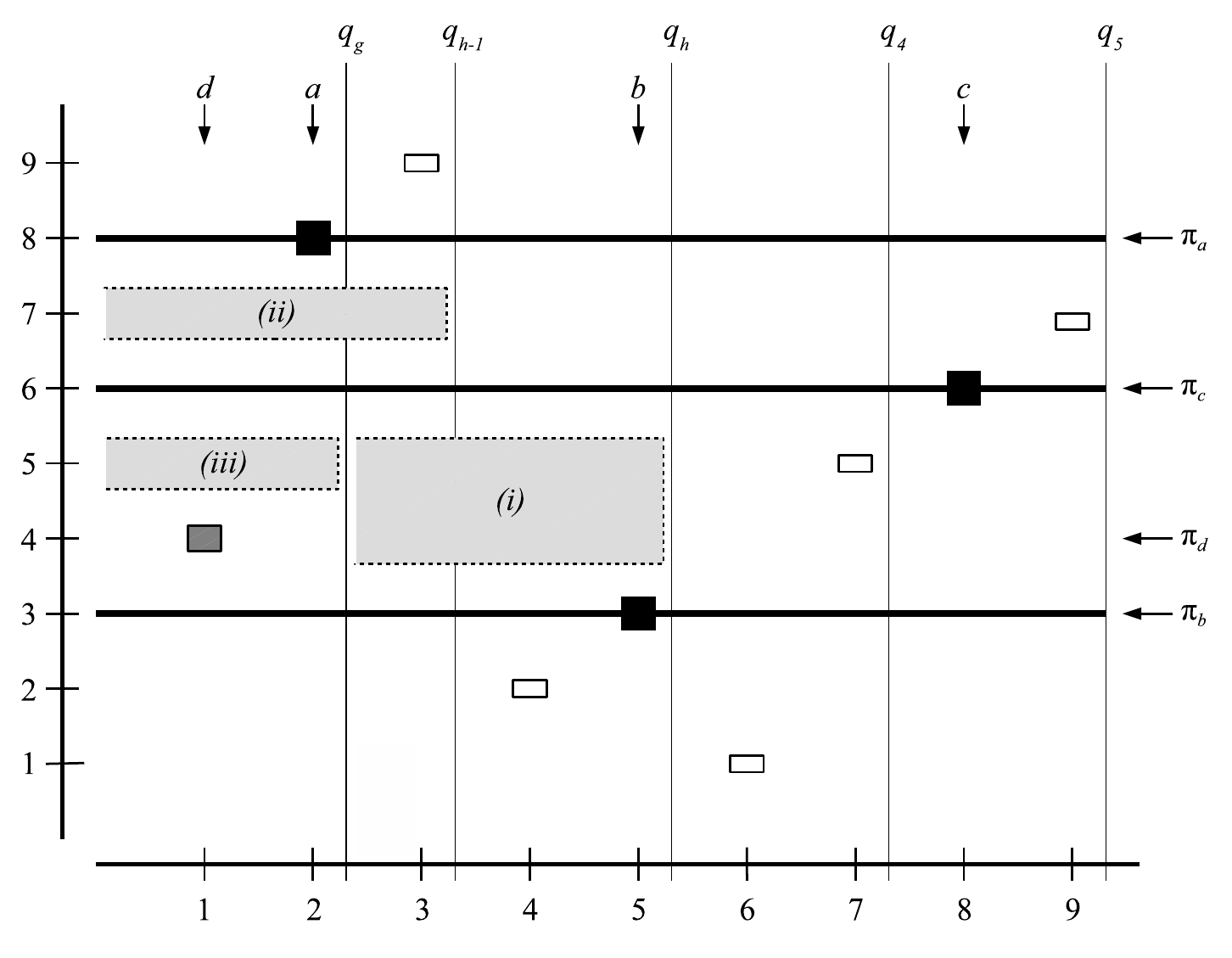}\label{fig71}
    \caption*{Figure 7.1.  Prohibited regions (i), (ii), and (iii) for $\pi = (4,8;9;2,3;1,5;6,7)$.}
  \end{center}
\end{figure}

Set $Y := Y_\lambda(\pi)$.  Now let $\chi$ be the permutation resulting from swapping the entry $\pi_b$ with the entry $\pi_d$ in $\pi$; so $\chi_b := \pi_d, \chi_d := \pi_b$, and $\chi_e := \pi_e$ when $e \notin \{ b, d \}$.  (If $d = b$, then $\chi = \pi$ with $\chi_b = \pi_b = \chi_d = \pi_d$.)  Let $\bar{\chi}$ be the $\lambda$-permutation produced from $\chi$ by sorting each cohort into increasing order.  Set $X := Y_\lambda(\bar{\chi})$.  Let $j$ denote the column index of the rightmost column with length $q_h$; so the value $\chi_b = \pi_d$ appears precisely in the $1^{st}$ through $j^{th}$ columns of $X$.  Let $f \leq h$ be such that $d \in (q_{f-1}, q_f]$, and let $k \geq j$ denote the column index of the rightmost column with length $q_f$.  The swap producing $\chi$ from $\pi$ replaces $\pi_d = \chi_b$ in the $(j+1)^{st}$ through $k^{th}$ columns of $Y$ with $\chi_d = \pi_b$ to produce $X$.  (The values in these columns may need to be re-sorted to meet the semistandard criteria.)  So $\chi_d \leq \pi_d$ implies $X \leq Y$ via a column-wise argument.

Forming the union of the prohibited rectangles for (i), (ii), and (iii), we see that there does not exist $e \in [q_{h-1}]$ such that $\pi_d = \chi_b < \pi_e < \pi_a$.  Thus we obtain:

\noindent (iv) For $l > j$, the $l^{th}$ column of $X$ does not contain any values from $[\chi_b, \pi_a)$.

\noindent Let $(j,i)$ denote the location of the $\chi_b$ in the $j^{th}$ column of $X$ (and hence $Y$).  So $Y_j(i) = \pi_d$.  By (iv) and the semistandard conditions, we have $X_{j+1}(u) = \pi_a$ for some $u \leq i$.  By (i) and (iii) we can see that $X_j(i+1) > \pi_c$.  Let $m$ denote the column index of the rightmost column of $\lambda$ with length $q_g$.  This is the rightmost column of $X$ that contains $\pi_a$.  Let $\mu \subseteq \lambda$ be the set of locations of the $\pi_a$'s in the $(j+1)^{st}$ through $m^{th}$ columns of $X$; note that $(j+1, u) \in \mu$.  Let $\omega$ be the permutation obtained by swapping $\chi_a = \pi_a$ with $\chi_b = \pi_d$ in $\chi$;  so $\omega_a := \chi_b = \pi_d$, $\omega_b := \chi_a = \pi_a$, $\omega_d := \chi_d = \pi_b$, and $\omega_e := \pi_e$ when $e \notin \{ d, a, b \}$.  Let $\bar{\omega}$ be the $\lambda$-permutation produced from $\omega$ by sorting each cohort into increasing order.  Set $W := Y_\lambda(\bar{\omega})$.  By (iv), obtaining $\omega$ from $\chi$ is equivalent to replacing the $\pi_a$ at each location of $\mu$ in $X$ with $\chi_b$ (and leaving the rest of $X$ unchanged) to obtain $W$.  So $\chi_b < \pi_a$ implies $W < X$.

Let $T$ be the result of replacing the $\pi_a$ at each location of $\mu$ in $X$ with $\pi_c$ (and leaving the rest unchanged).  So $T < X \leq Y$.  See the conceptual Figure 7.2 for $X$ and $T$; the shaded boxes form $\mu$.  In particular $T_{j+1}(u) = \pi_c$.  This $T$ is not necessarily a key; we need to confirm that it is semistandard.  For every $(q,p) \notin \mu$ we have $W_q(p) = T_q(p) = X_q(p)$.  By (iv), there are no values in $X$ in any column to the right of the $j^{th}$ column from $[\pi_c, \pi_a)$.  The region $\mu$ is contained in these columns.  Hence we only need to check semistandardness when moving from the $j^{th}$ column to $\mu$ in the $(j+1)^{st}$ column.  Here $u \leq i$ implies $T_j(u) \leq T_j(i) = \pi_d < \pi_c = T_{j+1}(u)$.  So $T \in \mathcal{T}_\lambda$.

\begin{figure}[h!]
  \begin{center}
    \includegraphics[scale=1]{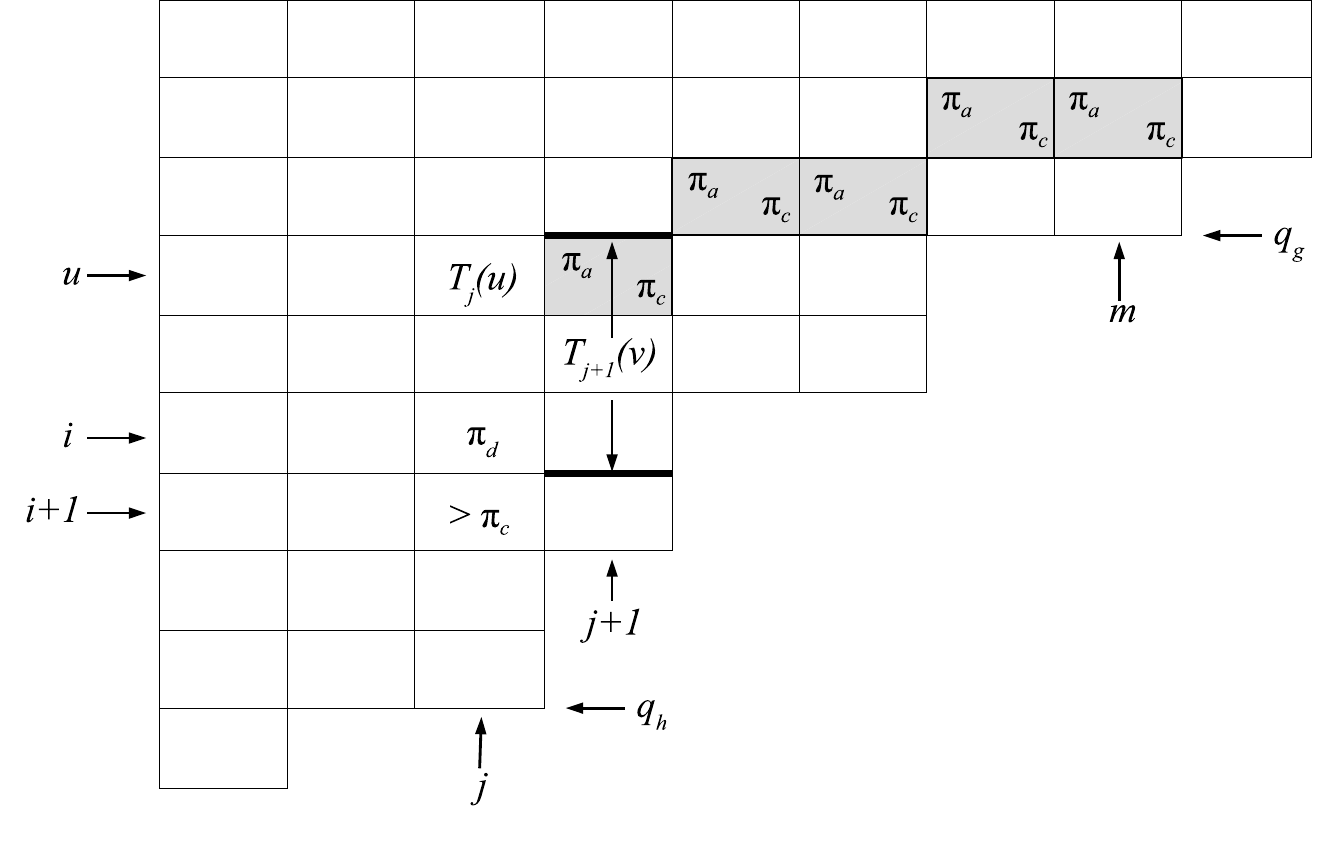}
    \caption*{Figure 7.2.  Values of $X$ (respectively $T$) are in upper left (lower right) corners.}
  \end{center}
\end{figure}

Now we consider the scanning tableau $S(T) =: S$ of $T$:  Since $(j, i+1) \notin \mu$, we have $T_j(i+1) = X_j(i+1)$.  Since $X_j(i+1) > \pi_c = T_{j+1}(u)$, the location $(j+1,u)$ is not in a scanning path $\mathcal{P}(T;j,i^\prime)$ for any $i^\prime > i$.  Since $T_j(i) = \chi_b = \pi_d < \pi_c$, the location $(j+1,v)$ is in $\mathcal{P}(T;j,i)$ for some $v \in [u,i]$.  By the semistandard column condition one has $T_{j+1}(v) \geq T_{j+1}(u) = \pi_c$.  Thus $S_j(i) \geq \pi_c > \chi_b = \pi_d = Y_j(i)$.  Hence $S(T) \nleq Y$, and so $T \notin \mathcal{D}_\lambda(\pi)$.  Since $T \in [Y]$, we have $\mathcal{D}_\lambda(\pi) \neq [Y]$.

In $\mathbb{R}^{|\lambda|}$, consider the line segment $U(t) = W + t(X-W)$, where $0 \leq t \leq 1$.  Here $U(0) = W$ and $U(1) = X$.  The value of $t$ only affects the values at the locations in $\mu$.  Let $x := \frac{\pi_c - \chi_b}{\pi_a - \chi_b}$.  Since $\chi_b < \pi_c < \pi_a$, we have $0 < x < 1$.  The values in $\mu$ in $U(x)$ are $\chi_b + \frac{\pi_c - \chi_b}{\pi_a-\chi_b}(\pi_a-\chi_b) = \pi_c$.  Hence $U(x) = T$.  Since $X$ and $W$ are keys, by Fact \ref{fact420}(ii) we have $S(X) = X$ and $S(W) = W$.  Then $W < X \leq Y$ implies $W \in \mathcal{D}_\lambda(\pi)$ and $X \in \mathcal{D}_\lambda(\pi)$.  Thus $U(0), U(1) \in \mathcal{D}_\lambda(\pi)$ but $U(x) \notin \mathcal{D}_\lambda(\pi)$.  If a set $\mathcal{E}$ is a convex polytope in $\mathbb{Z}^N$ and $U(t)$ is a line segment with $U(0), U(1) \in \mathcal{E}$, then $U(t) \in \mathcal{E}$ for any $0 < t < 1$ such that $U(t) \in \mathbb{Z}^N$.  Since $0 < x < 1$ and $U(x) = T \in \mathbb{Z}^{|\lambda|}$ with $U(x) \notin \mathcal{D}_\lambda(\pi)$, we see that $\mathcal{D}_\lambda(\pi)$ is not a convex polytope in $\mathbb{Z}^{|\lambda|}$.  \end{proof}

When one first encounters the notion of a Demazure polynomial, given Facts \ref{fact420}(iii)(iv) it is natural to ask when $\mathcal{D}_\lambda(\pi;x)$ is simply all of the ideal $[Y_\lambda(\pi)]$.  Since principal ideals in $\mathcal{T}_\lambda$ are convex polytopes in $\mathbb{Z}^{|\lambda|}$, we can answer this question while combining Theorems \ref{theorem420} and \ref{theorem520}:

\begin{cor}\label{cor520}Let $\pi \in S_n^\lambda$. The set $\mathcal{D}_\lambda(\pi)$ of Demazure tableaux of shape $\lambda$ is a convex polytope in $\mathbb{Z}^{|\lambda|}$ if and only if $\pi$ is $\lambda$-312-avoiding if and only if $\mathcal{D}_\lambda(\pi) = [Y_\lambda(\pi)]$.  \end{cor}

\noindent When $\lambda$ is the strict partition $(n,n-1, ..., 2,1)$, this convexity result appeared as Theorem 3.9.1 in \cite{Wi1}.

\section{Potential applications of convexity}

In addition to providing the core content needed to prove the main results of \cite{PW3}, our convexity results might later be useful in some geometric or representation theory contexts.  Our re-indexing of the $R$-312-avoiding phenomenon with gapless $R$-tuples could also be useful.   Fix $R \subseteq [n-1]$; inside $G := GL_n(\mathbb{C})$ this determines a parabolic subgroup $P := P_R$.   If $R = [n-1]$ then $P$ is the Borel subgroup $B$ of $G$.  Fix $\pi \in S_n^R$; this specifies a Schubert variety $X(\pi)$ of the flag manifold $G \slash P$.

Pattern avoidance properties for $\pi$ have been related to geometric properties for $X(\pi)$:  If $\pi \in S_n$ is 3412-avoiding and 4231-avoiding, then the variety $X(\pi) \subseteq G \slash B$ is smooth by Theorem 13.2.2.1 of \cite{LR}.  Since a 312-avoiding $\pi$ satisfies these conditions, its variety $X(\pi)$ is smooth.  Postnikov and Stanley \cite{PS} noted that Lakshmibai called these the ``Kempf'' varieties.  It could be interesting to extend the direct definition of the notion of Kempf variety from $G \slash B$ to all $G \slash P$, in contrast to using the indirect definition for $G \slash P$ given in \cite{HL}.

Berenstein and Zelevinsky \cite{BZ} emphasized the value of using the points in convex integral polytopes to describe the weights-with-multiplicities of representations.  Fix a partition $\lambda$ of some $N \geq 1$ such that $R_\lambda = R$.  Rather than using the tableaux in $\mathcal{T}_\lambda$ to describe the irreducible polynomial character of $G$ with highest weight $\lambda$ (Schur function of shape $\lambda$), the corresponding Gelfand-Zetlin patterns (which have top row $\lambda$) can be used.  These form an integral polytope in $\mathbb{Z}^{n \choose 2}$ that is convex.  In Corollary 15.2 of \cite{PS}, Postnikov and Stanley formed convex polytopes from certain subsets of the GZ patterns with top row $\lambda$; these had been considered by Kogan.  They summed the weights assigned to the points in these polytopes to obtain the Demazure polynomials $d_\lambda(\pi;x)$ that are indexed by the 312-avoiding permutations.  The convex integral polytope viewpoint was used there to describe the degree of the associated embedded Schubert variety $X(\pi)$ in the full flag manifold $G \slash B$.

Now assume that $\lambda$ is strict.  Here $R = [n-1]$ and the $R$-312-avoiding permutations are the 312-avoiding permutations.  The referee of \cite{PW3} informed us that Kiritchenko, Smirnov, and Timorin generalized Corollary 15.2 of \cite{PS} to express \cite{KST} the polynomial $d_\lambda(\pi;x)$ for any $\pi \in S_n^\lambda$ as a sum over the points in certain faces of the GZ polytope for $\lambda$ that are determined by $\pi$.  Only one face is used exactly when $\pi$ is 312-avoiding.  At a glance it may appear that their Theorem 1.2 implies that the set of points used from the GZ integral polytope for $d_\lambda(\pi;x)$ is convex exactly when $\pi$ is 312-avoiding.  So that referee encouraged us to remark upon the parallel 312-avoiding phenomena of convexity in $\mathbb{Z}^N$ for the tableau set $\mathcal{D}_\lambda(\pi)$ and of convexity in $\mathbb{Z}^{n \choose 2}$ for the set of points in these faces.  But we soon saw that when $\lambda$ is small it is possible for the union of faces used for $d_\lambda(\pi;x)$ to be convex even when $\pi$ is not 312-avoiding.  See Section 12 of \cite{PW3} for a counterexample.  To obtain convexity, one must replace $\lambda$ by $m\lambda$ for some $m \geq 2$.  In contrast, our Corollary \ref{cor520} holds for all $\lambda$.

Postnikov and Stanley remarked that the convex polytope of GZ patterns in the 312-avoiding case was used by Kogan and Miller to study the toric degeneration formed by Gonciulea and Lakshmibai for a Kempf variety.  It would be interesting to see if the convexity characterization of the $R$-312-avoiding Demazure tableau sets $\mathcal{D}_\lambda(\pi)$ found here is related to some nice geometric properties for the corresponding Schubert varieties $X(\pi)$ in $G \slash P$.  For any $R$-permutation $\pi$ the Demazure tableaux are well suited to studying the associated Schubert variety from the Pl{\"u}cker relations viewpoint, as was illustrated by Lax's re-proof \cite{Lax} of the standard monomial basis that used the scanning method of \cite{Wi2}.

\section{Parabolic Catalan counts}

The section (or paper) cited at the beginning of each item in the following statement points to the definition of the concept:

\begin{thm}\label{theorem18.1}Let $R \subseteq [n-1]$.  Write the elements of $R$ as $q_1 < q_2 < ... < q_r$.  Set $q_0 := 0$ and $q_{r+1} := n$.  Let $\lambda$ be a partition $\lambda_1 \geq \lambda_2 \geq ... \geq \lambda_n \geq 0$ whose shape has the distinct column lengths $q_r, q_{r-1}, ... , q_1$.  Set $p_h := q_h - q_{h-1}$ for $1 \leq h \leq r+1$.  The number $C_n^R =: C_n^\lambda$ of $R$-312-avoiding permutations is equal to the number of:

\noindent (i) \cite{GGHP}:  ordered partitions of $[n]$ into blocks of sizes $p_h$ for $1 \leq h \leq r+1$ that avoid the pattern 312, and $R$-$\sigma$-avoiding permutations for $\sigma \in \{ 123, 132, 213, 231, 321 \}$.

\noindent (ii)  Section 2:  multipermutations of the multiset $\{ 1^{p_1}, 2^{p_2}, ... , (r+1)^{p_{r+1}} \}$ that avoid the pattern 231.

\noindent (iii)  Section 2:  gapless $R$-tuples $\gamma \in UG_R(n)$.

\noindent (iv)  Here only:  $r$-tuples $(\mu^{(1)}, ... , \mu^{(r)})$ of shapes such that $\mu^{(h)}$ is contained in a $p_h \times (n-q_h)$ rectangle for $1 \leq h \leq r$ and for $1 \leq h \leq r-1$ the length of the first row in $\mu^{(h)}$ does not exceed the length of the $p_{h+1}^{st}$ (last) row of $\mu^{(h+1)}$ plus the number of times that (possibly zero) last row length occurs in $\mu^{(h+1)}$.

\noindent (v)  Sections 4 and 5:  $R$-rightmost clump deleting chains and gapless $\lambda$-keys.

\noindent (vi)  Section 6:  sets of Demazure tableaux of shape $\lambda$ that are convex in $\mathbb{Z}^{|\lambda|}$.

\end{thm}

\begin{proof}Part (i) first restates our $C_n^R$ definition with the terminology of \cite{GGHP}; for the second claim see the discussion below.  The equivalence for (ii) was noted in Section 2.  Use Proposition \ref{prop320.2}(ii) to confirm (iii).  For (iv), destrictify the gapless $R$-tuples within each carrel. Use Proposition \ref{prop320.2}(i) and Theorem \ref{theorem340} to confirm (v).  Part (vi) follows from Corollary \ref{cor520} and Fact \ref{fact420}(v).\end{proof}

To use the Online Encyclopedia of Integer Sequences \cite{Slo} to determine if the counts $C_n^R$ had been studied, we had to form sequences.  One way to form a sequence of such counts is to take $n := 2m$ for $m \geq 1$ and $R_m := \{ 2, 4, 6, ... , 2m-2 \}$.  Then the $C_{2m}^R$ sequence starts with 1, 6, 43, 352, 3114, ... ; this beginning appeared in the OEIS in Pudwell's A220097.  Also for $n \geq 1$ define the \emph{total parabolic Catalan number $C_n^\Sigma$} to be $\sum C_n^R$, sum over $R \subseteq [n-1]$.  This sequence starts with 1, 3, 12, 56, 284, ... ; with a `1' prepended, this beginning appeared in Sloane's A226316.  These ``hits'' led us to the papers \cite{GGHP} and \cite{CDZ}.

Let $R$ be as in the theorem.  Let $2 \leq t \leq r+1$.  Fix a permutation $\sigma \in S_t$.  Apparently for the sake of generalization in and of itself with new enumeration results as a goal, Godbole, Goyt, Herdan and Pudwell defined \cite{GGHP} the notion of an ordered partition of $[n]$ with block sizes $b_1, b_2, ... , b_{r+1}$ that avoids the pattern $\sigma$.  It appears that that paper was the first paper to consider a notion of pattern avoidance for ordered partitions that can be used to produce our $R$-312-avoiding permutations:  Take $b_1 := q_1$, $b_2 := q_2 - q_1$, ... , $b_{r+1} := n - q_r$, $t := 3$, and $\sigma := (3;1;2)$.  Their Theorem 4.1 implies that the number of such ordered partitions that avoid $\sigma$ is equal to the number of such ordered partitions that avoid each of the other five permutations for $t = 3$.  This can be used to confirm that the $C_{2m}^R$ sequence defined above is indeed Sequence A220097 of the OEIS (which is described as avoiding the pattern 123).  Chen, Dai, and Zhou gave generating functions \cite{CDZ} in Theorem 3.1 and Corollary 2.3 for the $C_{2m}^R$ for $R = \{ 2, 4, 6, ... , 2m-2 \}$ for $m \geq 0$ and for the $C_n^\Sigma$ for $n \geq 0$.  The latter result implies that the sequence A226316 indeed describes the sequence $C_n^\Sigma$ for $n \geq 0$.  Karen Collins and the second author of this paper have recently deduced that $C_n^\Sigma = \sum_{0 \leq k \leq [n/2]} (-1)^k \binom{n-k}{k} 2^{n-k-1} C_{n-k}$.

How can the $C_n^\Sigma$ total counts be modeled?  Gathering the $R$-312-avoiding permutations or the gapless $R$-tuples from Theorem \ref{theorem18.1}(ii) for this purpose would require retaining their ``semicolon dividers''.  Some other objects model $C_n^\Sigma$ more elegantly.  We omit definitions for some of the concepts in the next statement.  We also suspend our convention concerning the omission of the prefix `$[n-1]$-':  Before, a `rightmost clump deleting' chain deleted one element at each stage.  Now this unadorned term describes a chain that deletes any number of elements in any number of stages, provided that they constitute entire clumps of the largest elements still present plus possibly a subset from the rightmost of the other clumps.  When $n = 3$ one has $C_n^\Sigma = 12$.  Five of these chains were displayed in Section 6.  A sixth is \cancel{1} \cancel{2} \cancel{3}.  Here are the other six, plus one such chain for $n = 17$:

\begin{figure}[h!]
\begin{center}
\setlength\tabcolsep{.1cm}
\begin{tabular}{ccccc}
1& &2& &\cancel{3}\\
 &\cancel{1}& &\cancel{2}
\end{tabular}\hspace{7mm}
\begin{tabular}{ccccc}
1& &\cancel{2}& &3\\
 &\cancel{1}& &\cancel{3}
\end{tabular}\hspace{7mm}
\begin{tabular}{ccccc}
\cancel{1}& &2& &3\\
 &\cancel{2}& &\cancel{3}
\end{tabular}\hspace{7mm}
\begin{tabular}{ccccc}
1& &\cancel{2}& &\cancel{3}\\
 & & \cancel{1}& &
\end{tabular}\hspace{7mm}
\begin{tabular}{ccccc}
\cancel{1}& &2& &\cancel{3}\\
 & & \cancel{2}
\end{tabular}\hspace{7mm}
\begin{tabular}{ccccc}
\cancel{1}& &\cancel{2}& &{3}\\
 & &\cancel{3}
\end{tabular}
\end{center}
\end{figure}

\begin{figure}[h!]
\begin{center}
\setlength\tabcolsep{.3cm}
\begin{tabular}{ccccccccccccccccc}
1 & 2 & \cancel{3} & 4 & 5 & \cancel{6} & 7 & 8 & 9 & 10 & 11 & \cancel{12} & 13 & 14 & \cancel{15} & 16 & 17 \\
  &   & 1 & 2 & 4 & 5 & 7 & \cancel{8} & 9 & \cancel{10} & 11 & \cancel{13} & \cancel{14} & \cancel{16} & \cancel{17} & & \\
  &   &   &   &   & 1 & 2 & \cancel{4} & 5 & \cancel{7} & \cancel{9} & \cancel{11} & & & & & \\
  &   &   &   &   &   &   & \cancel{1} & \cancel{2} & \cancel{5} & & & & & & &
\end{tabular}
\end{center}
\end{figure}

\vspace{.1pc}\begin{cor}\label{cor18.2}  The total parabolic Catalan number $C_n^\Sigma$ is the number of:

\noindent (i)  ordered partitions of $\{1, 2, ... , n \}$ that avoid the pattern 312.

\noindent (ii)  rightmost clump deleting chains for $[n]$, and gapless keys whose columns have distinct lengths less than $n$.

\noindent (iii)  Schubert varieties in all of the flag manifolds $SL(n) / P_J$ for $J \subseteq [n-1]$ such that their ``associated'' Demazure tableaux form convex sets as in Section 7.  \end{cor}

\noindent Part (iii) highlights the fact that the convexity result of Corollary \ref{cor520} depends only upon the information from the indexing $R$-permutation for the Schubert variety, and not upon any further information from the partition $\lambda$.  In addition to their count $op_n[(3;1;2)] = C_n^\Sigma$, the authors of \cite{GGHP} and \cite{CDZ} also considered the number $op_{n,k}(\sigma)$ of such $\sigma$-avoiding ordered partitions with $k$ blocks.  The models above can be adapted to require the presence of exactly $k$ blocks, albeit of unspecified sizes.

\vspace{.5pc}\noindent \textbf{Added Note.}  We learned of the paper \cite{MW} after posting \cite{PW2} on the arXiv.  As at the end of Section 3, let $R$ and $J$ be such that $R \cup J = [n-1]$ and $R \cap J = \emptyset$.  It could be interesting to compare the definition for what we would call an `$R$-231-avoiding' $R$-permutation (as in \cite{GGHP}) to M{\"u}hle's and Williams' definition of a `$J$-231-avoiding' $R$-permutation in Definition 5 of \cite{MW}.  There they impose an additional condition $w_i = w_k + 1$ upon the pattern to be avoided.  For their Theorems 21 and 24, this condition enables them to extend the notions of ``non-crossing partition'' and of ``non-nesting partition'' to the parabolic quotient $S_n / W_J$ context of $R$-permutations to produce sets of objects that are equinumerous with their $J$-231-avoiding $R$-permutations.  Their Theorem 7 states that this extra condition is superfluous when $J = \emptyset$.  In this case their notions of $J$-non-crossing partition and of $J$-non-nesting partition specialize to the set partition Catalan number models that appeared as Exercises 159 and 164 of \cite{Sta}.  So if it is agreed that their reasonably stated generalizations of the notions of non-crossing and non-nesting partitions are the most appropriate generalizations that can be formulated for the $S_n / W_J$ context, then the mutual cardinality of their three sets of objects indexed by $J$ and $n$ becomes a competitor to our $C_n^R$ count for the name ``$R$-parabolic Catalan number''.  This development has made the obvious metaproblem more interesting:  Now not only must one determine whether each of the 214 Catalan models compiled in \cite{Sta} is ``close enough'' to a pattern avoiding permutation interpretation to lead to a successful $R$-parabolic generalization, one must also determine which parabolic generalization applies.

\vspace{1pc}\noindent \textbf{Acknowledgments.}  We thank Keith Schneider, Joe Seaborn, and David Raps for some helpful conversations, and we are also indebted to David Raps for some help with preparing this paper.  We thank the referee for suggesting some improvements in the exposition.

\end{document}